\documentclass[letter,final,oneside,notitlepage,10pt]{article}

\usepackage{amssymb}
\usepackage{amsmath}
\usepackage{amstext}
\usepackage[]{geometry}
\usepackage{mathrsfs}
\usepackage{euscript}
\usepackage[all]{xy}
\usepackage{xstring}
\usepackage{slashed}

\def\Z {\mathbb{Z}}

\def\R {\mathbb{R}}

\def\im{\mathrm{i}}
\def\id{\mathrm{id}}
\def\h {\mathrm{H}}

\def\trivlin{\mathbf{I}}

\def\quand{\quad\text{ and }\quad}

\def\quomma{\quad\text{, }\quad}

\def\quith{\quad\text{ with }\quad}

\def\ev{\mathrm{ev}}

\def\hc#1{\mathrm{h}_{#1}}

\def\stackerf#1#2{\stackrel{\text{\erf{#1}}}{#2}}

\def\nobr{~\hspace{-0.26em}}
\def\maps{\nobr:\nobr}
\def\df{\nobr := \nobr}
\def\eq{\nobr = \nobr}
\def\pr{{\mathsf{pr}}}

\let\Oldin\in\renewcommand{\in}{\nobr\Oldin\nobr}
\let\Oldtimes\times\renewcommand{\times}{\nobr\Oldtimes}
\let\Oldotimes\otimes\renewcommand{\otimes}{\nobr\Oldotimes}

\newlength{\widthtmp}
\def\length#1{\settowidth{\widthtmp}{#1}\the\widthtmp}

\renewcommand{\varepsilon}{\epsilon}
\makeatletter
\def\bigset#1#2{\left\lbrace\;\begin{minipage}[c]{#1}\begin{center}#2\end{center}\end{minipage}\;\right\rbrace}

\makeatother

\def\erf#1{(\ref{#1})}

\newlength{\myl}
\newcommand\sheaf[1]{\unitlength 0.1mm
  \settowidth{\myl}{$#1$}
  \addtolength{\myl}{-0.8mm}
  \begin{picture}(0,0)(0,0)
  \put(3,6){\text{\uline{\hspace{\myl}}}}
  \end{picture}#1\hspace{-0.15mm}}
\newcommand{\ueinssheaf}{\sheaf\ueins}

\newcommand{\ueins}{{\mathrm{U}}(1)}

\newcommand{\spin}[1]{{\mathrm{Spin}}\brackets{#1}}

\newcommand{\so}[1]{{\mathrm{SO}}\brackets{#1}}

\def\diff{\mathcal{D}\!i\!f\!\!f}

\def\triv#1{\mathcal{T}\!riv(#1)}
\def\trivcon#1{\mathcal{T}\!riv^{\nabla\!}(#1)}

\def\brackets#1{\IfStrEq{#1}{-}{}{(#1)}}
\def\subindex#1{\IfStrEq{#1}{-}{}{_{#1}}}
\def\buntech#1#2{\mathcal{B}\hspace{-0.01em}un_{\hspace{-0.1em}#1}^{#2}}

\def\grbtech#1{\mathcal{G}\hspace{-0.06em}r\hspace{-0.06em}b_{\hspace{-0.07em}{#1}}}
\def\grb#1#2{\grbtech#1\brackets{#2}}

\def\grbcon#1#2{\grbtech{#1}^{\nabla\!}\brackets{#2}}

\def\ugrb#1{\grb{\,}{#1}}

\newcommand{\alxydim}[2]{\begin{aligned}\xymatrix#1{#2}\end{aligned}}

\renewcommand{\to}{\nobr\!\xymatrix@R=0cm@C=1.4em{\ar[r] &}\nobr}
\renewcommand{\mapsto}{\!\xymatrix@R=0cm@C=1.4em{\ar@{|->}[r] &}\!}
\renewcommand{\Rightarrow}{\!\xymatrix@R=0cm@C=1.4em{\ar@{=>}[r] &}\!}
\renewcommand{\Leftarrow}{\!\xymatrix@R=0cm@C=1.4em{\ar@{<=}[r] &}\!}
\newcommand{\incl}{\!\xymatrix@R=0cm@C=1.4em{\ar@{^(->}[r] &}\!}

\renewcommand\Leftrightarrow{\!\xymatrix@R=0cm@C=1.4em{\ar@{<=>}[r] &}\!}

\usepackage{amsthm}
\usepackage{graphicx}
\usepackage[margin=2cm,font=normalsize,labelfont=bf]{caption}
\usepackage{verbatim}
\usepackage{enumerate}
\usepackage{fancyhdr}
\usepackage[]{geometry}
\usepackage[normalem]{ulem}
\usepackage{xstring}
\usepackage{needspace}

\makeatletter

\newcounter{denseversion}
\newcounter{authorcounter}
\newcounter{adresscounter}

\def\title#1{\gdef\@title{#1}}
\def\@title{}
\def\subtitle#1{\gdef\@subtitle{#1}}
\def\@subtitle{}

\def\authortagsused{0}
\def\adresstag#1{\if!#1!\else$^{\;#1\;}$\fi}

\renewcommand{\author}[2][]{
  \stepcounter{authorcounter}
  \if!#1!\else\gdef\authortagsused{1}\fi
  \ifnum\value{authorcounter}=1
    \def\@authorstringa{#2\adresstag{#1}}
    \def\@authorstringb{#2}
    \def\@authorstringc{#2\adresstag{#1}}
  \else
    \g@addto@macro\@authorstringa{\ and #2\adresstag{#1}}
    \g@addto@macro\@authorstringb{\ and #2}
    \g@addto@macro\@authorstringc{\\#2\adresstag{#1}}
  \fi}
\def\@author{\ifnum\value{denseversion}=0\@authorstringa\else\@authorstringb\fi}

\def\@adressstringa{}
\def\@adressstringb{}
\newcommand{\adress}[2][]{
  \stepcounter{adresscounter}
  \ifnum\value{adresscounter}=1
    \g@addto@macro\@adressstringa{\ifnum\authortagsused=0\def\br{\\}\else\def\br{, }\fi\adresstag{#1}#2}
    \g@addto@macro\@adressstringb{\def\br{\\}\adresstag{#1}\parbox[t]{14cm}{#2}}
  \else
    \g@addto@macro\@adressstringa{\\[\bigskipamount]\adresstag{#1}#2}
    \g@addto@macro\@adressstringb{\\[\medskipamount]\adresstag{#1}\parbox[t]{14cm}{#2}}
  \fi}
\def\@adress{\ifnum\value{denseversion}=0\@adressstringa\else\@adressstringb\fi}

\def\preprint#1{\gdef\@preprint{#1}}
\def\@preprint{}
\def\keywords#1{\gdef\@keywords{#1}}
\def\@keywords{}
\def\msc#1{\gdef\@msc{#1}}
\def\@msc{}
\def\email#1{
   \gdef\@email{#1}
   \g@addto@macro\@authorstringc{ {\it (#1)}}}
\def\@email{}
\def\dedication#1{\gdef\@dedication{#1}}
\def\@dedication{}

\def\mybaselinestretch#1{\gdef\@mybaselinestretch{#1}}
\def\@mybaselinestretch{}

\def\refname{References}

\mybaselinestretch{1.2}
\renewcommand{\baselinestretch}{\@mybaselinestretch}

\def\denseversion{
  \setcounter{denseversion}{1}
  \newgeometry{left=3cm,right=3cm,top=3cm}
  \mybaselinestretch{1.1}
  \renewcommand{\baselinestretch}{\@mybaselinestretch}
  \normalfont
  \fancyfoot[C]{\itshape{\hspace{2.5cm}--$\,\,$\thepage$\,\,$--}}}

\newlength{\myparskip}
\newlength{\myproofparskip}

\renewcommand{\emph}[1]{\def\reserved@a{it}\ifx\f@shape\reserved@a\uline{#1}\else\textit{#1}\fi}

\newcommand{\mytableofcontents}{
   \ifnum\value{denseversion}=0
     \tableofcontents
   \else
     \renewcommand{\baselinestretch}{0.8}
     \normalfont
     \tableofcontents
     \renewcommand{\baselinestretch}{\@mybaselinestretch}
     \normalfont
   \fi}

\newlength{\zeilenlaenge}
\def\putindent#1{
  \settowidth{\zeilenlaenge}{#1}
  \ifnum\zeilenlaenge>\textwidth
    #1
  \else
    \noindent #1
  \fi
}

\def\href#1#2{#2}

\def\kohyp{
  \usepackage{hyperref}
  \hypersetup{
    linktocpage = true,
    pdftitle = {\@title},
    pdfauthor = {\@author},
    pdfkeywords = {\@keywords},    
    bookmarksopen = true,
    bookmarksopenlevel = 1
  }}  
\def\showkeywords{\begin{flushleft}\footnotesize\textbf{Keywords}: \@keywords\end{flushleft}}
\def\showmsc{\begin{flushleft}\footnotesize\textbf{MSC 2010}: \@msc\end{flushleft}}

\newcounter{mythm}[subsection]
\newcounter{mainthm}

\def\setsecnumdepth#1{
  \setcounter{secnumdepth}{#1}
  \setcounter{mythm}{0}
  \ifnum \c@secnumdepth >0
    \ifnum \c@secnumdepth >1
      \def\themythm{\thesubsection.\arabic{mythm}}
      \numberwithin{equation}{subsection}
      \renewcommand\theequation{\thesubsection.\arabic{equation}}
    \else
      \def\themythm{\thesection.\arabic{mythm}}
      \numberwithin{equation}{section}
      \renewcommand\theequation{\thesection.\arabic{equation}}
    \fi
  \else
    \def\themythm{\arabic{mythm}}
  \fi}

\newenvironment{mythmenv}{\strut\ \setlength{\parskip}{\myproofparskip}}{\setlength{\parskip}{\myparskip}}

\newlength{\mythmskip}
\newlength{\mythmtopskip}
\newtheoremstyle{mythmstylea}{\mythmtopskip}{\mythmskip}{\it}{}{\bf}{.}{0em}{}
\newtheoremstyle{mythmstyleb}{\mythmtopskip}{\mythmskip}{}{}{\bf}{.}{0em}{}

\theoremstyle{mythmstylea}
\newtheorem{mytheorem}[mythm]{Theorem}
\newtheorem{mydefinition}[mythm]{Definition}
\newtheorem{mycorollary}[mythm]{Corollary}
\newtheorem{myproposition}[mythm]{Proposition}
\newtheorem{mylemma}[mythm]{Lemma}

\newtheorem{mymaintheorem}[mainthm]{Theorem}
\newtheorem{mymaincorollary}[mainthm]{Corollary}
\newtheorem{mymainproposition}[mainthm]{Proposition}
\newtheorem{mymaindefinition}[mainthm]{Definition}

\theoremstyle{mythmstyleb}
\newtheorem{myremark}[mythm]{Remark}
\newtheorem{myexample}[mythm]{Example}
\newtheorem{myexercise}[mythm]{Exercise}

\newenvironment{theorem}[1][]{\begin{mytheorem}[#1]\begin{mythmenv}}{\end{mythmenv}\end{mytheorem}}
\newenvironment{definition}[1][]{\begin{mydefinition}[#1]\begin{mythmenv}}{\end{mythmenv}\end{mydefinition}}
\newenvironment{corollary}[1][]{\begin{mycorollary}[#1]\begin{mythmenv}}{\end{mythmenv}\end{mycorollary}}
\newenvironment{proposition}[1][]{\begin{myproposition}[#1]\begin{mythmenv}}{\end{mythmenv}\end{myproposition}}
\newenvironment{lemma}[1][]{\begin{mylemma}[#1]\begin{mythmenv}}{\end{mythmenv}\end{mylemma}}
\newenvironment{remark}[1][]{\begin{myremark}[#1]\begin{mythmenv}}{\end{mythmenv}\end{myremark}}
\newenvironment{example}[1][]{\begin{myexample}[#1]\begin{mythmenv}}{\end{mythmenv}\end{myexample}}

\renewenvironment{proof}[1][Proof]{\noindent #1. \begin{mythmenv}}{\hfill$\square$\end{mythmenv}\medskip}

\def\mytitle{}
\def\zmptitle{
  \begin{tabular}{cc}
    \begin{minipage}[c]{0.4\textwidth}
      \begin{flushleft}
        \includegraphics[width=110pt]{../../tex/zmp}
      \end{flushleft}  
    \end{minipage}&
    \begin{minipage}[c]{0.55\textwidth}
      \begin{flushright}
      {\small\sf\@preprint}
      \end{flushright}
    \end{minipage}
  \end{tabular}
  \vskip 2cm}

\def\maketitle{
  \setlength{\parskip}{\myparskip}  
  \newpage
  \noindent
  \mytitle
  \begin{center}
    \LARGE\@title\\
    \if!\@subtitle!\else \smallskip\LARGE\@subtitle\\\fi
    \bigskip
    \if!\@author!\else\bigskip\large\@author\\\fi
    \ifnum\value{denseversion}=0
      \if!\@adress!\else     \bigskip\normalsize\@adress\\\fi
      \if!\@email!\else\ifnum\value{authorcounter}=1\bigskip\normalsize\textit{\@email}\\\else\fi\fi
    \else
    \fi
    \if!\@dedication!\else \bigskip\normalsize{\@dedication}\\\fi
  \end{center}
  \ifnum\value{denseversion}=0\vskip 1.5cm\else\vskip0.5cm\fi
  \thispagestyle{empty}}

\def\kobiburl#1{
   \IfBeginWith
     {#1}
     {http://arxiv.org/abs/}
     {\kobibarxiv{#1}}
     {\kobiblink{#1}}}
\def\kobibarxiv#1{\href{#1}{\texttt{[arxiv:\StrGobbleLeft{#1}{21}]}}}
\def\kobiblink#1{Available as: \href{#1}{\texttt{\StrSubstitute{#1}{_}{\underline{\;\;}}}}}

\newcommand{\etalchar}[1]{$^{#1}$}
\def\kobib#1{
  \begin{raggedright}
  \ifnum\value{denseversion}=0\else\small\fi

  \end{raggedright}
  \ifnum\value{denseversion}=0\else
      \noindent
      \if!\@authorstringc!\else
        \ifnum\authortagsused=0\ifnum\value{authorcounter}>1\normalsize\@authorstringc\\[\medskipamount]\else\fi\else\normalsize\@authorstringc\\[\medskipamount]\fi
      \fi
      \if!\@adress!\else\normalsize\@adress\\{}\fi
      \ifnum\authortagsused=0
        \ifnum\value{authorcounter}=1
          \if!\@email!\else\linebreak\normalsize\textit{\@email}\\{}\fi
        \else
        \fi
      \else
      \fi
  \fi
  }

\newenvironment{commentfigure}{}
\newenvironment{sidewayscommentfigure}{\begin{minipage}}{\end{minipage}}

\def\showcomments{ -- Comments suppressed}

\newif\if@fewtab\@fewtabtrue{
  \count255=\time\divide\count255 by 60
  \xdef\hourmin{\number\count255}
  \multiply\count255 by-60\advance\count255 by\time
  \xdef\hourmin{\hourmin:\ifnum\count255<10 0\fi\the\count255}}
\def\ps@draft{
  \let\@mkboth\@gobbletwo
  \def\@oddfoot{
    \hbox to 7 cm{\tiny \versionno\hfil}
    \hskip -7cm\hfil\rm\thepage\hfil{\tiny\draftdate}}
  \def\@oddhead{}
  \def\@evenhead{}
  \let\@evenfoot\@oddfoot}
\def\draftdate{\number\month/\number\day/\number\year\ \ \ \hourmin }
\newcommand\version[1]{
  \typeout{}\typeout{#1}\typeout{}
  \vskip-1.7cm \centerline{\fbox{{\normalsize\tt DRAFT -- #1 -- 
  \draftdate\showcomments}}} \vskip0.92cm}
\def\draft#1{
  \def\versionno{#1}
  \pagestyle{draft}\thispagestyle{draft}
  \gdef\@ntitle{\version\versionno \@title}
  \global\def\draftcontrol{1}}
\global\def\draftcontrol{0}

\makeatother

\def\p{P}

\def\ev{\mathrm{ev}}

\def\hc#1{\mathrm{h}_{#1}}
\def\pcomp{\nobr\star\nobr}

\def\prev#1{\overline{#1}}

\def\un{\mathscr{R}}

\def\tr{\mathscr{T}}

\def\fusbun#1#2{\mathcal{F}\!us\buntech#1{}(#2)}

\def\ufusbun#1{\mathcal{F}\!us\buntech{}{}(#1)}

\def\struc#1#2{#1\text{-}\mathcal{L}\!i\!f\!t(#2)}

\usepackage[latin1]{inputenc}
\usepackage[english]{babel}

\def\quot#1{``#1''}

\title{\Large Spin structures on loop spaces that characterize string manifolds}
\author{Konrad Waldorf}
\adress{Fakultät für Mathematik\br Universität Regensburg\\ Universitätsstraße 31\br D--93053 Regensburg}
\email{konrad.waldorf@mathematik.uni-regensburg.de}

\kohyp

\usepackage{amstext}
\usepackage[normalem]{ulem}

\def\lspin#1{L\spin #1}
\def\lspinhat#1{\widetilde{L\spin #1}}

\def\inf#1{\EuScript{#1}}
\def\trivfus#1{\mathcal{F}\!u\!s\!\triv{#1}}
\def\fusstruc#1#2{#1\text{-}\mathcal{F}\!u\!s\mathcal{L}i\!f\!t(#2)}

\def\lop{\cup\,}
\def\generator{\gamma_{can}}
\def\wzwmodel{\inf L}

\def\diffr{\delta}
\def\gbas{\mathcal{G}_{bas}}
\def\adjust#1{\!\!\!\begin{tabular}{c}$#1$\end{tabular}\!\!\!}

\usepackage{amsmath}
\usepackage{amssymb}

\begin{document}


\maketitle

\begin{abstract}
Classically, a spin structure on the loop space of a  manifold is  a lift of the structure group of the looped  frame bundle from the loop group  to its universal central extension. Heuristically, the loop space of a  manifold is spin if and only if the manifold itself is a string manifold, against which it is well-known that only the if-part is   true in general. In this article we develop a new version of spin structures on loop spaces that exists if and only if the manifold is string, as desired. This new version consists of a classical spin structure plus a certain fusion product related to loops of frames in the manifold. We use the lifting gerbe theory of Carey-Murray, recent results of Stolz-Teichner on loop spaces,   and some own results about string geometry and Brylinski-McLaughlin transgression.

\end{abstract}

\mytableofcontents


\setsecnumdepth{1}

\section{Introduction}

\label{sec:motivation}

The Witten genus has been introduced by Witten using supersymmetric sigma models with target space a spin manifold $M$, and $S^1$-equivariant Dirac operators on the free loop space $LM$ \cite{witten2}. Although the Witten genus is  well-defined  (a power series in the Pontryagin numbers of $M$) the approach via loop space geometry still lacks a rigorous understanding.

Dirac operators can be considered on manifolds with a  spin structure, i.e. with a lift of the structure group of the frame bundle to its universal covering group. The frame bundle of the loop space of an $n$-dimensional spin manifold $M$ is an $\lspin n$-bundle; Killingback defined \cite{killingback1} a spin structure on $LM$ to be a lift of its structure group to the universal central extension
\begin{equation}
\label{eq:centralextension0}
1 \to \ueins \to \lspinhat n \to \lspin n \to 1\text{.}
\end{equation} 
Killingback showed that the existence of spin structures on $LM$ is obstructed by a class $\lambda_{LM}\in \h^3(LM,\Z)$. This class is often called the \emph{string class of $M$}; however, in view of the following discussion we better call it  the \emph{spin class of $LM$}.

Killingback showed\footnote{In fact, Killingback defers the proof to a paper \quot{in preparation} that I was not able to find. Since that time, the statement seems to be \quot{well-known}, and I do not  know what the earliest reference for a proof is.  Just to be sure, one proof is given in \cite[Section 7]{Nikolausa}; see Theorem \ref{th:liftingtransgression} in the present paper.  } that the class $\lambda_{LM}$ is the transgression of the first fractional Pontryagin class of $M$, $\frac{1}{2}p_1(M) \in \h^4(M,\Z)$, i.e. the image of $\frac{1}{2}p_1(M)$ under the homomorphism
\begin{equation*}
\tau\maps \h^4(M,\Z) \to \h^3(LM,\Z):x \mapsto \int_{S^1} \ev^{*}x \text{,}
\end{equation*}
where $\ev: S^1 \times LM \to M$ is the evaluation map.
Further results about the relation between the classes $\lambda_{LM}$ and $\frac{1}{2}p_1(M)$ have been obtained by Pilch-Warner \cite{Pilch1988} (see below),  Carey-Murray \cite{carey7} (for the based loop space), and McLaughlin \cite{mclaughlin1} (for simply-connected manifolds).

Spin manifolds with vanishing first fractional Pontryagin class,  $\frac{1}{2}p_1(M)=0$, are called \emph{string manifolds}. Given the relation
\begin{equation*}
\textstyle\tau(\frac{1}{2}p_1(M)) = \lambda_{LM}
\end{equation*}
it is clear that the loop space of a string manifold in  spin. In this article we are concerned with the converse proposition: is a manifold string when its loop space is spin? 
We recall three  seminal results concerning this question.
With methods of algebraic topology,  McLaughlin showed the following.

\begin{theorem}[{{\cite{mclaughlin1}}}]
Suppose $M$ is a 2-connected spin manifold of dimension greater than $5$. Then, $LM$ is spin if and only if $M$ is string. 
\end{theorem}

Similarly, but with more advanced methods using Hochschild cohomology, Kuribayashi and Yamaguchi\- proved that the assumption of 2-connectedness may be replaced by a  condition that admits  non-trivial $\pi_2$. 

\begin{theorem}[{{\cite{Kuribayashi1998}}}]
Let $M$ be a simply-connected smooth manifold. Suppose $M$ is 4-dimensional, or $M$ has the structure of  a compact homogenous  space, or $M$ is a product of such. Then, $LM$ is spin if and only if $M$ is string. 
\end{theorem}

In contrast to the previous two results,  Pilch and Warner have shown in pioneering work that the assumption of simply-connectedness cannot not be dropped  (at least not when instead of the frame bundle  a general principal $\spin n$-bundle $P$ is considered).

\begin{theorem}[{{\cite{Pilch1988}}}]
\label{th:pilchwarner}
There exists a non-simply connected smooth manifold $M$ and a principal $\spin n$-bundle $P$ over $M$ such that $\frac{1}{2}p_1(P)\neq 0$ and $LP$ is spin. \end{theorem}

The situation that the loop space of a spin manifold $M$ is spin while $M$ is not string occurs evidently  if the transgression homomorphism $\tau$ is not injective. Loosely speaking,  the passage to the loop space \emph{loses information}. The question is how this lost information can be restored on the loop space side. On a geometrical level, this means to add additional structure to spin structures on loop spaces. And just to come back to the Dirac operators on the loop space: such operators could then be required to \emph{preserve} this additional structure.

Witten proposes that spin structures on the loop space have to be \emph{equivariant} with respect to the rotation action of $S^1$ on $LM$, and correspondingly considers $S^1$-equivariant Dirac operators  \cite{witten2}. This leads to  $S^1$-equivariant index theory on loop spaces; see, e.g.,  \cite{Alvarez1987,Alvarez1987a}. The addition of $S^1$-equivariance, or more general, equivariance under the group $\diff^{+}(S^1)$ of orientation-preserving diffeomorphisms of $S^1$, indeed eliminates the  counterexample of Theorem \ref{th:pilchwarner}, as proved in \cite{Pilch1988}. In general, however, it is, to my best knowledge, not known whether a manifold is string if and only if its loop space has a $\diff^{+}(S^1)$-equivariant spin structure.

One of the problems with $\diff^{+}(S^1)$-equivariance is that  $\diff^{+}(S^1)$ is connected and hence acts separately on each connected component of $LM$. Assuming for a moment that $M$ is connected, these components are labelled by the fundamental group $\pi_1(M)$, so that the obstruction  against lifting a spin structure on $LM$ to a $\diff^{+}(S^1)$-equivariant spin structure splits into $|\pi_1(M)|$ many unrelated obstructions. 

In this article we introduce a new additional structure for spin structures on loop spaces that in particular establishes a relation between  the separate spin structures on different connected components of $LM$.  This new additional structure is called \emph{fusion product}, and the corresponding spin structures are called \emph{fusion spin structures}; see Definition \ref{def:fusionspinstructure}. For the convenience of the reader let me summarize the main idea of a fusion product. 

 In a more general context, fusion products  are defined for $\ueins$-bundles over loop spaces \cite{waldorf10}. A fusion product defines a relation between the fibres of the bundle over the three loops that emerge from three paths connecting a common initial point with a common end point. In particular, since these three loops may be elements of different connected components of the loop space, the existence of a fusion product cannot be explored separately over each connected component. 

In order to apply the general concept of a fusion product to the present situation, we make two observations. The first is that the central extension $\lspinhat n$, considered as a $\ueins$-bundle over $\lspin n$, carries a canonical fusion product. In fact, we prove in Theorem \ref{th:fusionextension} that these canonical fusion products exist for a large class of central extensions of loop groups, including the universal central extension of a compact, connected, simple and simply-connected Lie group. 

The second observation is that a spin structure on the loop space $LM$ can equivalently be understood as a $\ueins$-bundle $\inf S$ over $LFM$, the loop space of the total space of the frame bundle of $M$, together with a certain  action of $\lspinhat n$. Now, our new additional structure for spin structures on loop spaces is a fusion product on $\inf S$, subject to the condition that the $\lspinhat n$-action on $\inf S$ is compatible with the two fusion products on $\lspinhat n$ and $\inf S$; see Definition \ref{def:fusionspinstructure}. 

The main result of this article is the following theorem.

\begin{theorem}
\label{th:main}
Let $M$ be a spin manifold of dimension $n=3$ or $n>4$. Then, $M$ is string if and only if $LM$ is fusion spin. 
\end{theorem}

This paper is organized as follows. The subsequent Section \ref{sec:orientations} is complementary and concerns the   problem of characterizing \emph{spin manifolds} by \emph{orientations} of  loop spaces, as well as its solution by Stolz and Teichner. This problem is analogous to the one addressed in the present article but in \quot{one degree lower}. The content of Section \ref{sec:orientations} is not used in the main text, but might be useful for getting familiar with the topic.

Section \ref{sec:fusionspinstructures} gives the definition of our new notion of fusion spin structures (Definition \ref{def:fusionspinstructure}), and introduces and reviews the necessary material. In particular, we reveal the canonical fusion product on the central extension $\lspinhat n$ using  an explicit model of $\lspinhat n$ motivated by conformal field theory and introduced by Mickelsson \cite{mickelsson1}.

The next two sections prepare the machinery for the proof of our main result. Section \ref{sec:generaltheory} is concerned with the loop space side. There we recast fusion spin structures in  the context of lifting bundle gerbes. The theory of lifting bundle gerbes has been invented by Murray \cite{Murray} for ordinary central extensions of Lie groups. The main result of Section \ref{sec:generaltheory} is an extension of this theory to \emph{fusion} extensions, i.e. central extensions of loop groups by $\ueins$ whose underlying principal bundle  carries a fusion product.   It includes a new additional structure for bundle gerbes over loop spaces called an \emph{internal fusion product} (Definition \ref{def:internalfusionproduct}). Lifting bundle gerbes with internal fusion products  serve as a bundle gerbe-theoretic setting for fusion spin structures.

In Section \ref{sec:transgressionstringstructures} we provide a similar, gerbe-theoretical setting for string structures, on the basis of my paper \cite{waldorf8}. That is, we regard string structures as trivializations of a certain  bundle 2-gerbe, the Chern-Simons 2-gerbe \cite{carey4}. As the main point in Section \ref{sec:transgressionstringstructures} we introduce a  categorical version of the transgression homomorphism $\tau$, which takes a bundle 2-gerbe over $M$ to a bundle gerbe over $LM$ with internal fusion product. We prove that the transgression of the Chern-Simons 2-gerbe gives the spin lifting bundle gerbe with its internal fusion product (Proposition \ref{prop:fusionpreserving}). This geometrical transgression procedure establishes  the relation between string structures on $M$ and fusion spin structures on $LM$. In Section \ref{sec:conclusion} we assemble the pieces and prove Theorem \ref{th:main}.

I remark that we 
\emph{do not discuss} the relation between the \emph{set} of string structures on  $M$ and the \emph{set} of fusion spin structures on $LM$. Although (according to Theorem \ref{th:main}) one set is empty if and only if the other is empty, fusion spin structures are still not good enough to achieve a \emph{bijection} between the two sets. Such bijection is the subject of ongoing research, and will additionally employ a certain equivariance under thin homotopies of loops. 
As a consequence, we can -- at the moment -- not provide any new insights to the theory of Dirac operators on the loop space.

\paragraph{Acknowledgements.} This work is supported by the Deutsche Forschungsgemeinschaft (DFG) within the scientific network \quot{String Geometry} (project code 594335). I thank the Department of Mathematics at Hamburg University for kind hospitality and support during the summer term. Many thanks to André Henriques, Thomas Nikolaus, Ulrich Pennig, Peter Teichner,  Chris Schommer-Pries, Urs Schreiber, and Christoph Wockel  for helpful questions, comments, and discussions.

\section{Spin Manifolds and Loop Space Orientations}

\label{sec:orientations}

Analogous to the question how spin structures on loop spaces can characterize string manifolds is the question how \emph{orientations} of loop spaces can characterize \emph{spin manifolds}.  This question has been considered early by Atiyah \cite{atiyah2} and recently solved by Stolz and Teichner \cite{stolz3}.

The following canonical double covering $\inf O_{LM}$ of $LM$ is considered as the \emph{orientation bundle} of the loop space. The monodromy in the central extension
\begin{equation*}
1 \to \Z_2 \to \spin n\to \so n \to 1
\end{equation*}
is a smooth map $m\maps L\so n \to \Z_2$. Taking free loops in the oriented frame bundle $FM$ of $M$ produces an $L\so n$-bundle $LFM$ over $LM$, and $\inf O_{LM}$ is obtained by extending the structure group of $LFM$ along $m$, i.e.
\begin{equation*}
\inf O_{LM} \df LFM \times_{L\so n} \Z_2\text{.}
\end{equation*}
Accordingly, an \emph{orientation} of $LM$ is defined to be a section of $\inf O_{LM}$.

On the level of cohomology classes, the class in $\h^1(LM,\Z_2)$ of $\inf O_{LM}$ is the transgression of the second Stiefel-Whitney class $w_2(M) \in \h^2(M,\Z_2)$, i.e. the image of $w_2(M)$ under the $\Z_2$-reduced transgression homomorphism
\begin{equation*}
\tau \maps  \h^2(M,\Z_2) \to \h^{1}(LM,\Z_2)\text{.}
\end{equation*}
In particular, $LM$ is orientable if $M$ is spin.

Atiyah remarked \cite{atiyah2} that for a simply-connected manifold $M$ the converse statement is true, so that then $M$ is spin if and only if $LM$ is orientable. This statement was later proved in detail by McLaughlin \cite[Proposition 2.1]{mclaughlin1}.

The problem of characterizing non-simply connected spin manifolds was solved by Stolz and Teichner \cite{stolz3}. Using methods of spin geometry they recognized a crucial additional structure on the orientation bundle $\inf O_{LM}$, a so-called \emph{fusion product}. Accordingly, among all orientations of $LM$, there is a subclass consisting of \emph{fusion-preserving orientations}. Stolz and Teichner showed:

\begin{theorem}[{{\cite[Theorem 9]{stolz3}}}]
Let $M$ be an oriented Riemannian manifold. Then, there is a bijection
\begin{equation*}
\bigset{9em}{Fusion-preserving orientations of $LM$}
\cong
\bigset{7em}{Equivalence classes of spin structures on $M$}\text{.}
\end{equation*}
In particular, $M$ is spin if and only if $LM$ is fusion orientable.
\end{theorem}

The notion of a fusion product is also crucial for the new version of spin structures that we introduce in this article, and is explained in the following section.

\setsecnumdepth 1

\section{Fusion Spin Structures}

\label{sec:fusionspinstructures}
\label{sec:spinstructures}
\label{sec:spinliftinggerbe}

In this section we explain our new version of spin structures on loop spaces, which we call \emph{fusion spin structure}, see Definition \ref{def:fusionspinstructure}. We consider a spin manifold $M$ of dimension $n=3$ or $n > 4$, in which case the group $\spin n$ is simple, connected, simply-connected and compact. We  fix a generator $\generator\in\h^3(\spin n,\Z) \cong \Z$. The loop group $\lspin n$ has a  universal central extension \cite{pressley1}
\begin{equation}
\label{eq:centralextension}
\alxydim{}{1 \ar[r] & \ueins \ar[r] & 
\begin{minipage}[b]{4em}
$\lspinhat n$\vspace{0.4em}
\end{minipage} \ar[r]^-{p} & \lspin n \ar[r] & 1\text{.}}
\end{equation}
This universal extension is determined up to a sign, which we fix by requiring the following identity for its first Chern class (when considered as a principal  $\ueins$-bundle over $L\spin n$):
\begin{equation}
\label{eq:signconvention}
\tau(\generator) + \mathrm{c}_1(\lspinhat n) = 0 \in \h^2(L\spin n,\Z)\text{.}
\end{equation}

 We start by reviewing the classical notion of spin structures on loop spaces. 
We denote by $FM$  the spin-oriented frame bundle of $M$, which is a $\spin n $-principal bundle over $M$. Since $\spin n$ is connected, $LFM$ is a principal $\lspin n$-bundle over $LM$, see \cite[Lemma 5.1]{waldorf13} and \cite[Proposition 1.9]{spera1}. 

\begin{definition}[{{\cite{killingback1}}}]
A \emph{spin structure on $LM$} is a lift of the structure group of the looped  frame bundle $LFM$ from $\lspin n$ to the central extension $\lspinhat n$.
\end{definition}

Thus, a spin structure on $LM$ is a pair $(\inf S,\sigma)$ of a principal $\lspinhat n$-bundle $\inf S$ over $LM$ together with a smooth map $\sigma\maps \inf S \to LFM$ such that the diagram
\begin{equation*}
\alxydim{@R=0.8cm}{\inf S \times \lspinhat n \ar[dd]_{\sigma \times p} \ar[r] & \inf S \ar[dd]_{\sigma} \ar[dr] \\ &&LM \\ LFM \times \lspin n \ar[r] & LFM \ar[ur] }
\end{equation*}
is commutative. A morphism between spin structures $(\inf S_1,\sigma_1)$ and $(\inf S_2,\sigma_2)$ is a bundle morphism $\varphi\maps\inf S_1 \to \inf S_2$ such that $\sigma_1=\sigma_2 \circ \varphi$. 
The following lemma is a general fact in lifting theory, which we need later.
\begin{lemma}
\label{lem:ueinsbundle}
Suppose $(\inf S,\sigma)$ is a spin structure on $LM$.
Then,   $\sigma\maps  \inf S \to LFM$ together with the $\ueins$-action on $\inf S$ induced along the group homomorphism
\begin{equation*}
\alxydim{}{\ueins \ar[r]^-{i} &  \ueins \ar@{^(->}[r] & \begin{minipage}[b]{4em}
$\lspinhat n$\text{,}\vspace{0.25em}
\end{minipage}}
\end{equation*}
where $i$ denotes the inversion of the abelian group $\ueins$,
is a principal $\ueins$-bundle over $LM$.  
\end{lemma}

An important role in this article is played by so-called \emph{fusion} in loop spaces. In order to explain fusion in a  general context, let $X$ be a connected smooth manifold. By $PX$ we denote the set of paths in $X$ with \quot{sitting instants}, i.e., smooth maps $\gamma\maps [0,1] \to X$ that are locally constant near the endpoints.  We denote by $PX^{[k]}$ the $k$-fold fibre product of $PX$ over the evaluation map $\ev\maps PX \to X \times X$, i.e. the set of $k$-tuples of paths with a common initial point and a common end point. We have a map
\begin{equation*}
\lop\maps  \p X^{[2]} \to L X\maps  (\gamma_1,\gamma_2) \mapsto \prev{\gamma_2} \pcomp \gamma_1\text{,}
\end{equation*}
where $\pcomp$ denotes the path concatenation, and $\prev{\gamma}$ denotes the reversed path. This map is well-defined (it produces smooth loops) due to the sitting instants of the paths.

\begin{definition}[{{\cite[Definition 2.1.3]{waldorf10}}}]
\label{def:fusionproduct}
Let $A$ be an abelian Lie group, and let $\inf E$ be a Fréchet principal $A$-bundle over $LX$.
A \emph{fusion product} on $\inf E$ assigns to each triple $(\gamma_1,\gamma_2,\gamma_3) \in PX^{[3]}$  a smooth, $A$-equivariant map
\begin{equation*}
\lambda_{\gamma_1,\gamma_2,\gamma_3}\maps  \inf E_{\gamma_1 \lop \gamma_2}  \otimes \inf E_{\gamma_2\lop\gamma_3} \to \inf E_{\gamma_1 \lop \gamma_3}\text{,}
\end{equation*} 
such that the following two conditions are satisfied:
\begin{enumerate}[(i)]

\item 
Associativity: for $(\gamma_1,\gamma_2,\gamma_3,\gamma_4) \in PX^{[4]}$ and $q_{ij} \in \inf E_{\gamma_i \lop \gamma_j}$, 
\begin{equation*}
\lambda_{\gamma_1,\gamma_3,\gamma_4}(\lambda_{\gamma_1,\gamma_2,\gamma_3}(q_{12} \otimes q_{23}) \otimes q_{34})= \lambda_{\gamma_1,\gamma_2,\gamma_4}(q_{12} \otimes \lambda_{\gamma_2,\gamma_3,\gamma_4}(q_{23} \otimes q_{34}) )\text{.}
\end{equation*}

\item
Smoothness: if $U$ is a smooth manifold and $c\maps U \to PX^{[3]}$ is a map for which the three induced maps $e_{ij} \df \lop \circ \pr_{ij} \circ c\maps  U \to LX$ are all smooth, then
\begin{equation*}
\lambda_c\maps  e_{12}^{*}\inf E \otimes e_{23}^{*}\inf E \to e_{13}^{*}\inf E
\end{equation*}
is a smooth morphism between bundles over $U$.
\end{enumerate}
\end{definition}

I remark that in my papers \cite{waldorf9,waldorf10,waldorf11} I have treated the smoothness of fusion products in the convenient setting of diffeological spaces; above definition is equivalent but completely avoids diffeological spaces.

\begin{definition}
\label{def:fusionextension}
Let $G$ be a Lie group and let $A$ be an abelian Lie group. A \emph{fusion extension of $LG$ by $A$} is a central extension
\begin{equation*}
1 \to A \to \inf G \to LG \to 1
\end{equation*}
of Fréchet Lie groups, together with a  multiplicative fusion product $\lambda$ on the principal $A$-bundle $\inf G$. 
\end{definition}

Here, a \emph{multiplicative} fusion product is one such that
\begin{equation}
\label{eq:multfus}
\lambda_{\gamma_1,\gamma_2,\gamma_3}(q_{12} \otimes q_{23}) \cdot \lambda_{\gamma_1',\gamma_2',\gamma_3'}(q_{12}' \otimes q_{23}') = \lambda_{\gamma_1\gamma_1',\gamma_2\gamma_2',\gamma_3\gamma_3'}(q^{}_{12}q_{12}' \otimes q^{}_{23}q_{23}')
\end{equation}
for all elements $q_{ij}\in \inf G_{\gamma_i\lop \gamma_j}$ and $q'_{ij}\in \inf G_{\gamma_i' \lop \gamma_j'}$ and all $(\gamma_1,\gamma_2,\gamma_3),(\gamma_1',\gamma_2',\gamma_3') \in PG^{[3]}$. 

Fusion extensions are relevant for the present article because of the following fact:
\begin{theorem}
\label{th:fusionextensionobserve}
The universal central extension $\lspinhat n$ is a fusion extension (in a canonical way). 
\end{theorem}

We will give two proofs of this theorem. The first proof is given in the next paragraphs: we construct an explicit model $\wzwmodel$ for the central extension $\lspinhat n$ and exhibit its fusion product. The second proof appears in  Section \ref{sec:fusionextensions} as Corollary \ref{co:fusionuniversal}; there we construct another  model that is less explicit but explains  better and from a more general context why this fusion product is present. There we  also show that the two constructions are canonically isomorphic (Proposition \ref{prop:equivext}).

Our explicit model $\wzwmodel$ is motivated by conformal field theory and has been introduced by Mickelsson \cite{mickelsson1}. It exists for any 2-connected Lie group $G$, such as $\spin n$.
We consider pairs $(\phi,z)$ where $\phi\maps D^2 \to G$ is a smooth map and $z\in \ueins$. For technical reasons, we require that $\phi$ is radially constant near the boundary, i.e.  there exists  $\varepsilon>0$ such that $\phi(r \mathrm{e}^{2\pi\im\varphi})=\phi(\mathrm{e}^{2\pi\im\varphi})$ for all  $1-\varepsilon<r\leq 1$. 
On the set of pairs we impose the following equivalence relation:
\begin{equation*}
(\phi,z) \sim (\phi',z') 
\quad\Leftrightarrow\quad
\partial\phi = \partial\phi'
\quad\text{ and }\quad
z = z' \cdot \mathrm{e}^{2\pi\im S_{\mathrm{WZ}}(\Phi)}\text{.}
\end{equation*}
Here, $\partial\phi$ denotes the restriction of $\phi$ to the boundary, and  $\Phi\maps  S^2 \to \spin n$ is the  map defined on the northern hemisphere by $\phi$ (with the orientation-preserving identification) and on the southern hemisphere by $\phi'$ (with the orientation-reversing identification). Due to the above technical assumptions, this gives a smooth map.
The symbol  $S_{\mathrm{WZ}}$ stands for the \emph{Wess-Zumino term}, which is defined as follows. Because $G$ is 2-connected, the map $\Phi$ can be extended to a smooth map $\tilde\Phi\maps  D^3 \to G$ defined on the solid ball. Then,
\begin{equation}
\label{eq:H}
S_{\mathrm{WZ}}(\Phi) \df \int_{D^3} \tilde\Phi^{*}H
\quith 
H\df\textstyle\frac{1}{6} \left \langle \theta \wedge [\theta \wedge \theta]  \right \rangle \in \Omega^3(G)\text{.}
\end{equation}
Here, $\theta \in \Omega^1(G,\mathfrak{g})$ is the left-invariant Maurer-Cartan from on $G$. The bilinear form $\left \langle -,-  \right \rangle$ is  normalized such that the closed 3-form $H$ represents the fixed generator $\generator \in \h^3(\spin n,\Z)$. Now, the total space of the principal $\ueins$-bundle $\wzwmodel$ of our model is the set of equivalence classes of above pairs:
\begin{equation*}
\wzwmodel \df \left \lbrace (\phi,z) \right \rbrace \;/\;\sim\text{.}
\end{equation*} 
The bundle projection sends $(\phi,z)$  to $\partial\phi\in LG$, and the $\ueins$-action is given by multiplication in the $\ueins$-component.

The group structure on $\wzwmodel$ turning it into a central extension is given by the  \emph{Mickelsson product}   \cite{mickelsson1}:
\begin{equation*}
\wzwmodel \times \wzwmodel \to \wzwmodel\maps  ((\phi_1, z_1), (\phi_2, z_2)) \mapsto (\phi_1\phi_2, z_1z_2 \cdot \exp \left ( -2\pi\im \int_{D^2} (\phi_1,\phi_2)^{*}\rho \right )  )  \text{,}
\end{equation*}
where $\rho$ is defined by
\begin{equation}
\label{eq:rho}
\textstyle
\rho \df  \frac{1}{2}\left \langle  \pr_1^{*}\theta \wedge \pr_2^{*}\bar\theta  \right \rangle \in \Omega^2(G \times G)\text{.}
\end{equation}
The two differential forms $H$ and $\rho$, which are the only parameters of the construction, satisfy the identities
\begin{equation*}
H_{g_1g_2} = H_{g_1}  + H_{g_2}- \mathrm{d}\rho_{g_1,g_2}
\quand
\rho_{g_1,g_2}  + \rho_{g_1g_2,g_3}  =  \rho_{g_2,g_3} +  \rho_{g_1,g_2g_3}
\end{equation*}
for all $g_1,g_2,g_3\in G$. 
The first identity assures that the Mickelsson product is well-defined on equivalence classes, and the second implies its associativity.

Now we come to the fusion product. For $(\gamma_1,\gamma_2,\gamma_3) \in PG^{[3]}$, we define
\begin{equation}
\label{eq:explicitfusionproduct}
\begin{array}{rcccl}
\lambda_{\gamma_1,\gamma_2,\gamma_3} &\maps&  \wzwmodel_{\gamma_1\lop\gamma_2} \otimes \wzwmodel_{\gamma_2\lop\gamma_3} &\to& \wzwmodel_{\gamma_1\lop\gamma_3}
\\ &&(\phi_{12},z_{12}) \otimes (\phi_{23},z_{23}) &\mapsto& (\phi_{13},z_{12}  z_{23}\cdot \mathrm{e}^{2\pi\im S_{\mathrm{WZ}}(\Psi)})\text{,} 
\end{array}
\end{equation}
where $\phi_{13}\maps  D^2 \to G$ is an arbitrarily chosen smooth map with  $\partial\phi_{13}=\gamma_1 \lop \gamma_3$, and $\Psi\maps  S^2 \to G$ is obtained by trisecting $S^2$ along the longitudes $0$, $\frac{2\pi}{3}$ and $\frac{4\pi}{3}$, and  prescribing $\Psi$ on each sector with the maps $\phi_{12}$, $\phi_{23}$ (with orientation-reversing identification) and $\phi_{13}$ (with orientation-preserving identification). This map $\Psi$ is  smooth due to the sitting instants of the paths and the requirement that the maps $\phi_{ij}$ are radially constant.

That definition \erf{eq:explicitfusionproduct} of the fusion product on $\wzwmodel$ is independent of the choice of $\phi_{13}$ follows from the identity $ S_{\mathrm{WZ}}(\Psi) =  S_{\mathrm{WZ}}(\Psi')S_{\mathrm{WZ}}(\Phi_{13})$ for Wess-Zumino terms, where $\Psi'$ is obtained as described above but using a different map $\phi_{13}'$ instead of $\phi_{13}$, and $\Phi_{13}$ is obtained in the way described earlier from $\phi_{13}$ and $\phi_{13}'$. 
Definition \erf{eq:explicitfusionproduct} is also well-defined under the equivalence relation $\sim$ due to a similar identity for Wess-Zumino terms.  
Finally, it is associative in the sense of Definition \ref{def:fusionproduct}.

Now that we have explained that $\lspinhat n$ is a fusion extension, we proceed with introducing our new version of a spin structure on $LM$. 

\begin{definition}
\label{def:fusionspinstructure}
A \emph{fusion spin structure} on $LM$ is a spin structure $(\inf S, \sigma)$ together with a fusion product $\lambda_{\inf S}$ on the associated principal $\ueins$-bundle $\sigma\maps  \inf S \to LFM$ of Lemma \ref{lem:ueinsbundle}, such that the $\lspinhat n$-action on $\inf S$ is fusion-preserving:
\begin{equation*}
\lambda_{\inf S}(q_{12}\cdot \beta_{12} \otimes q_{23} \cdot \beta_{23}) = \lambda_{\inf S}(q_{12} \otimes q_{23}) \cdot \lambda(\beta_{12} \otimes \beta_{23})\text{,}
\end{equation*}
where $\lambda$ is the fusion product of the fusion extension $\lspinhat n$, and
 $q_{12},q_{23} \in \inf S$, $\beta_{12},\beta_{23} \in \lspinhat n$ are supposed to be such that the fusion products are defined. 
\end{definition}

More explicitly, the condition for the elements  means that there exist paths $(\alpha_1,\alpha_2,\alpha_3) \in PFM^{[3]}$ and $(\gamma_1,\gamma_2,\gamma_3) \in P\spin n^{[3]}$ such that $\sigma(q_{ij}) = \alpha_i \lop \alpha_j$ and $p(\beta_{ij})=\gamma_i \lop \gamma_j$.

\setsecnumdepth 1

\section{Lifting Gerbes for Fusion Extensions}

\label{sec:generaltheory}

The main objective of this section is to  embed the definition of a fusion spin structure into a more general theory. The results we derive in this setting will be used in the proof of the main result. 

\setsecnumdepth 2

\subsection{Lifting Gerbes}

\label{sec:liftinggerbes}

\label{sec:liftinggerbesbasics}

We briefly review the theory of lifting bundle gerbes for the convenience of the reader, following \cite{Murray}. 
The setup is a central extension
\begin{equation}
\label{ce}
\alxydim{}{1 \ar[r] & A \ar[r] & \adjust{\hat G} \ar[r]^{t} & G \ar[r] & 1}
\end{equation} 
of (possibly Fréchet)  Lie  groups, and a  principal $G$-bundle $P$ over a (possibly Fréchet) manifold $X$. 

A \emph{$\hat G$-lift of $P$} is a principal $\hat G$-bundle $\hat P$ over $X$ together with a bundle map $f\maps  \hat P \to P$ satisfying $f(\hat p \cdot\hat g) = f(\hat p) \cdot t(\hat g)$ for all $\hat p \in \hat P$ and $\hat g\in \hat G$. $\hat G$-lifts of $P$ form a category $\struc {\hat G}P$. The existence of $\hat G$-lifts is obstructed by a class  $\xi_P \in \h^2(X,\sheaf A) $ that is obtained by locally lifting a \v Cech cocycle for $P$ and then measuring the error in the cocycle condition over triple overlaps.

\begin{example}
\label{ex:spin1}
If $P = LFM$ is the looping of the frame bundle of a spin manifold, and  $\hat G = \lspinhat n$ is the universal central extension  of $G=L\spin n$, then the spin structures on $LM$  of Section \ref{sec:fusionspinstructures} form precisely the category $\struc {\lspinhat n}{LFM}$. The obstruction class $\xi_{LFM}\in \h^2(LM,\ueinssheaf)$ can be identified with a class in $\h^3(LM,\Z)$; this is the spin class  $\lambda_{LM}$ of $LM$ that was mentioned in Section \ref{sec:motivation}.
\end{example}

Associated to the given bundle $P$ is the following bundle gerbe $\mathcal{G}_P$ over $X$, called the \emph{lifting bundle gerbe} \cite{Murray}. Its surjective submersion is the bundle projection $\pi\maps  P \to X$. Over the two-fold fibre product $P^{[2]} \df P \times_X P$, the lifting bundle gerbe has the  principal $A$-bundle $Q \df \diffr^{*}\hat G$, obtained by regarding $\hat G$ as a principal $A$-bundle over $G$, and pulling it back along the \quot{difference map} 
\begin{equation}
\label{eq:differencemap}
\diffr\maps P^{[2]} \to G
\quith
p \cdot \diffr(p,p') = p'\text{.}
\end{equation}
Finally, the multiplication of $\hat G$ defines a bundle gerbe product, i.e., a bundle isomorphism
\begin{equation}
\label{bgprod}
\mu\maps  \pr_{12}^{*}Q \otimes \pr_{23}^{*}Q \to \pr_{13}^{*}Q\maps  (p_1,p_2,g_{12}),(p_2,p_3,g_{23}) \mapsto (p_1,p_3,g_{12}g_{23})
\end{equation}
over $P^{[3]}$ that is associative over $P^{[4]}$. Here, $\pr_{ij}\maps  P^{[3]} \to P^{[2]}$ denote the projections to the indexed components. The Dixmier-Douady class of the lifting gerbe $\mathcal{G}_P$ is the obstruction class $\xi_P$ \cite{Murray}. 

\begin{example}
\label{ex:spinliftingM}
If $P=FM$ is the oriented frame bundle of an oriented Riemannian manifold, and the central extension is
\begin{equation*}
1 \to \Z_2 \to \spin n\to \so n \to 1\text{,}
\end{equation*}
the associated lifting gerbe $\mathcal{G}_{FM}$ is the \emph{spin lifting gerbe} of $M$. Its Dixmier-Douady class $\xi_{FM} \in \h^2(M,\Z_2)$ is the second Stiefel-Whitney class $w_2$.
\end{example}

A \emph{trivialization} of a bundle gerbe $\mathcal{G}$ is an isomorphism $\mathcal{T}\maps  \mathcal{G} \to \mathcal{I}$, where $\mathcal{I}$ denotes the trivial bundle gerbe \cite{waldorf1}. Trivializations form a category that we denote by $\triv {\mathcal{G}}$.
In case of the lifting bundle gerbe $\mathcal{G}_P$, a trivialization is a principal $A$-bundle $T$ over $P$ together with a bundle isomorphism
\begin{equation}
\label{triviso}
\kappa\maps  Q \otimes \pr_2^{*}T \to \pr_1^{*}T
\end{equation}
over $P^{[2]}$ satisfying a  compatibility condition with the bundle gerbe product $\mu$, namely
\begin{equation}
\label{eq:actioncondition}
\kappa(q_{12} \otimes \kappa(q_{23} \otimes t)) = \kappa(\mu(q_{12} \otimes q_{23}) \otimes t)
\end{equation}
for all $(p_1,p_2,p_3) \in P^{[3]}$ and all $t \in T_{p_3}$, $q_{12} \in Q_{p_1,p_2}$, and $q_{23}\in Q_{p_2,p_3}$.

Suppose $(T,\kappa)$ is such a trivialization of $\mathcal{G}_P$. 
Then, $\hat P \df T$ with the projection $T \to P \to  X$ and  the $\hat G$-action $\hat p \cdot \hat g \df \kappa(\hat g^{-1} \otimes \hat p)$ is a principal $\hat G$-bundle over $X$, and  together with the bundle projection $f\maps  T \to P$ it is a $\hat G$-lift of $P$.
Conversely, suppose $f\maps \hat P \to P$ is  a $\hat G$-lift of $P$. Then $T \df \hat P$ equipped with the map $f\maps  T \to P$ and the $A$-action induced by
\begin{equation*}
\alxydim{}{A \ar[r]^{i} & A \ar@{^(->}[r] & \hat G\text{,}}
\end{equation*} 
where $i$ is the inversion of the group $A$,  is a principal $A$-bundle over $P$, and together with the bundle morphism $\kappa$ defined by $\kappa(\hat g \otimes  \hat p) \df \hat p \cdot \hat g^{-1}$ a trivialization of $\mathcal{G}_P$.
The main theorem of lifting bundle theory states that these two constructions are inverse to each other:

\begin{theorem}[{{\cite{Murray}}}]
\label{th:lifting}
Let $P$ be a principal $G$-bundle over $X$. Then,  above  constructions constitute an equivalence of categories,
\begin{equation*}
\triv {\mathcal{G}_P} \cong \struc {\hat G}P\text{.}
\end{equation*}
\end{theorem}

\begin{example}
\label{ex:spinliftinggerbe}
In the situation of Example \ref{ex:spin1}, the lifting bundle gerbe is denoted $\mathcal{S}_{LM}$ and called the \emph{spin lifting gerbe}. 
Its Dixmier-Douady class satisfies
\begin{equation*}
\mathrm{dd}(\mathcal{S}_{LM}) + \lambda_{LM} = 0\text{.}
\end{equation*} 
Theorem \ref{th:lifting} implies an equivalence:
\begin{equation*}
\bigset{7em}{Spin structures on $LM$}
\cong 
\bigset{10em}{Trivializations of the spin lifting gerbe $\mathcal{S}_{LM}$}\text{.}
\end{equation*}
Under this equivalence, a spin structure $(\inf S,\sigma)$ corresponds to a trivialization $(T,\kappa)$ of $\mathcal{S}_{LM}$ whose principal $\ueins$-bundle $T$ is the one of Lemma \ref{lem:ueinsbundle}.
\end{example}

\subsection{Transgression and Regression}
\label{sec:fusion}

In this section we explain the role of fusion products from a more general perspective. Based on Definition \ref{def:fusionproduct}, fusion bundles with structure group $A$ over the loop space $LX$ of a  smooth manifold $X$ form a category that we denote by $\fusbun A{LX}$. 
These categories are monoidal and natural with respect to looped maps, i.e., if $f\maps  X \to Y$ is a smooth map between diffeological spaces, and $Lf\maps  LX \to LY$ denotes the induced map on loop spaces, pullback is a functor
\begin{equation}
\label{eq:fusionpullback}
Lf^{*}\maps  \fusbun A{LY} \to \fusbun A{LX}\text{.}
\end{equation}

Basically, fusion bundles are those bundles over $LX$ that correspond to bundle gerbes over $X$. This can be seen in both ways. Firstly,  fusion products furnish a \emph{regression functor}
\begin{equation}
\label{eq:regressionfunctor}
\un_x\maps  \fusbun  A{LX} \to  \grb AX
\end{equation}
that lands in the 2-category of bundle gerbes over $X$, see \cite[Section 5.1]{waldorf10}. It is defined for connected manifolds $X$ and depends on the choice of a base point $x\in X$ (up to canonical, natural equivalence).
This functor is weak (i.e., it has non-trivial compositor 2-morphisms), but  induces a honest functor into the  category $\hc 1 \grb AX$ obtained from the 2-category $\grb AX$ by identifying 2-isomorphism 1-morphisms. We remark two evident properties:

\begin{lemma}
\label{lem:regprop}
The functor $\un_x$ has the following properties:
\begin{enumerate}[(i)]

\item 
it is monoidal.

\item
it is natural with respect to smooth, base-point-preserving maps.

\end{enumerate}
\end{lemma}

Secondly, in the other direction, there is a \emph{transgression functor}
\begin{equation}
\label{eq:transgression}
\tr\maps  \hc 1 \grbcon AX \to \fusbun A{LX}
\end{equation}
defined on the  category of $A$-bundle gerbes with connections over $X$.
It has been introduced by Brylinski and McLaughlin \cite{brylinski4} and lifted to fusion bundles in \cite[Section 4.2]{waldorf10}. On the level of characteristic classes, it satisfies
\begin{equation}
\label{eq:transgressionclasses}
\mathrm{c}_1(\tr_{\mathcal{G}}) = - \tau(\mathrm{dd}(\mathcal{G}))\text{.}
\end{equation}

We shall describe some details of the transgression functor following \cite{waldorf5,waldorf10}. 
If $\mathcal{G}$ is a bundle gerbe with connection over $X$, the fibre of $\tr_\mathcal{G}$ over a loop $\tau \in LX$ is
\begin{equation}
\label{eq:lgfibre}
\tr_\mathcal{G}|_{\tau} \df \hc 0 \trivcon{\tau^{*}\mathcal{G}}\text{,}
\end{equation} 
i.e. it consists of isomorphism classes of connection-preserving trivializations of $\tau^{*}\mathcal{G}$. A connection-preserving isomorphism $\mathcal{A}\maps \mathcal{G}_1 \to \mathcal{G}_2$ induces a  bundle morphism given by
\begin{equation}
\label{eq:transmorph}
\tr_{\mathcal{A}}\maps  \tr_{\mathcal{G}_1} \to \tr_{\mathcal{G}_2}\maps  \mathcal{T} \mapsto \mathcal{T} \circ \tau^{*}\mathcal{A}^{-1}\text{.}
\end{equation}

The lift of this construction into the category of fusion bundles over $LX$ is established by recognizing a fusion product $\lambda_{\mathcal{G}}$ on the bundle $\tr_{\mathcal{G}}$. 
Let us briefly recall how $\lambda_{\mathcal{G}}$ is characterized. We denote by $\iota_1,\iota_2\maps  [0,1] \to S^1$ the inclusion of the interval into the left and the right half of the circle. Let $(\gamma_1,\gamma_2,\gamma_3)$ be a triple of paths with a common initial point $x$ and a common end point $y$, and let $\mathcal{T}_{ij}$ be trivializations of the pullback of $\mathcal{G}$ to the loops $\gamma_i \lop \gamma_j$, for  $(ij)=(12),(23),(13)$. Then,
the relation
\begin{equation}
\label{eq:fusiondef}
\lambda_{\mathcal{G}}(\mathcal{T}_{12} \otimes \mathcal{T}_{23}) = \mathcal{T}_{13}
\end{equation}
holds if and only if there exist 2-isomorphisms
\begin{equation*}
\phi_1\maps  \iota_1^{*}\mathcal{T}_{12} \Rightarrow \iota_1^{*}\mathcal{T}_{13}
\quomma
\phi_2\maps  \iota_2^{*}\mathcal{T}_{12} \Rightarrow \iota_1^{*}\mathcal{T}_{23}
\quand
\phi_3\maps  \iota_2^{*}\mathcal{T}_{23} \Rightarrow \iota_2^{*}\mathcal{T}_{13}
\end{equation*}
between trivializations of the pullbacks of $\mathcal{G}$ to the paths $\gamma_1$, $\gamma_2$, and $\gamma_3$, respectively, such that their restrictions to the two common points $x$ and $y$ satisfy the cocycle condition
$\phi_1 = \phi_3 \circ \phi_2$.

\begin{example}
The orientation bundle $\inf O_{LM}$ of the loop space, which was mentioned in Section \ref{sec:orientations}, is the transgression of the spin lifting gerbe $\mathcal{G}_{FM}$ of $M$, see Example \ref{ex:spinliftingM}. This explains the existence of a fusion product on $\inf O_{LM}$, independently of its discovery by Stolz and Teichner \cite{stolz3}.
\end{example}

The regression functor $\un_x$ and the transgression functor $\tr$ are inverse to each other, in a way that has various formulations; for the purpose of this article we need:

\begin{theorem}
\label{th:transgressionregression}
The diagram
\begin{equation*}
\alxydim{@R=1.6cm}{ & \fusbun A{LX} \ar[dr]^-{\un_x} & \\ \hc 1 \grbcon AX \ar[ur]^-{\tr} \ar[rr] &&  \hc 1 \grb AX }
\end{equation*}
of functors, which has on the bottom the functor that forgets connections, is commutative up to a canonical natural equivalence. 
\end{theorem}

I remark that transgression and regression can be turned into a honest equivalence of categories, by either  including the connection on the  loop space side or dropping the connections on the gerbes, see the main results of  \cite{waldorf10,waldorf11}.

\subsection{Multiplicative Gerbes and Fusion Extensions}

\label{sec:fusionextensions}

In this section we  explain how  fusion extensions (Definition \ref{def:fusionextension}) can be obtained by transgression of multiplicative gerbes. Let $A$ be an abelian Lie group. We recall from  \cite[Definition 1.3]{waldorf5} that a \emph{multiplicative $A$-bundle gerbe with connection} over a  Lie group $G$ is a triple $(\mathcal{G},\rho,\mathcal{M},\alpha)$ consisting of an $A$-bundle gerbe $\mathcal{G}$ with connection over $G$  together with a 2-form $\rho\in \Omega^2(G\times G,\mathfrak{a})$ with values in the Lie algebra $\mathfrak{a}$ of $A$,  a connection-preserving isomorphism
\begin{equation}
\label{eq:multstr}
\mathcal{M}\maps  \pr_1^{*}\mathcal{G} \otimes \pr_2^{*}\mathcal{G}
\to m^{*}\mathcal{G} \otimes \mathcal{I}_{\rho}
\end{equation}
between gerbes over $G \times G$, and a connection-preserving transformation
$\alpha$
over $G \times G \times G$, that serves as an associator for the multiplication \erf{eq:multstr} and satisfies
the   pentagon axiom. In \erf{eq:multstr} we have denoted by $\mathcal{I}_{\rho}$ the trivial bundle gerbe equipped with the  curving 2-form $\rho$. That \erf{eq:multstr} is connection-preserving implies for $H\df\mathrm{curv}(\mathcal{G}) \in \Omega^3(G)$ the identity 
\begin{equation}
\label{eq:id1}
H_{g_1g_2} = H_{g_1}  + H_{g_2}- \mathrm{d}\rho_{g_1,g_2}
\end{equation}
and the existence of $\alpha$ implies the identity
\begin{equation}
\label{eq:id2}
\rho_{g_1,g_2}  + \rho_{g_1g_2,g_3}  =  \rho_{g_2,g_3} +  \rho_{g_1,g_2g_3}\text{.}
\end{equation}

Let us first explain how the central extension is produced. The transgression of $\mathcal{G}$ is a principal $A$-bundle $\inf G \df \tr_{\mathcal{G}}$ over $LG$. Next, the transgression of $\mathcal{M}$ is a bundle isomorphism
\begin{equation*}
\tr_{\mathcal{M}}\maps \pr_1^{*}\inf G \otimes \pr_2^{*}\inf G \to m^{*}\inf G \otimes \tr_{\mathcal{I}_{\rho}}\text{.}
\end{equation*}
As described in \cite[Section 3.1]{waldorf5} the $\ueins$-bundle $\tr_{\mathcal{I}_{\rho}}$ has a canonical trivialization
\begin{equation}
\label{eq:cantriv}
t_{\rho}\maps  \tr_{\mathcal{I}_{\rho}} \to \trivlin\text{.}
\end{equation}
Together, we obtain a bundle isomorphism
\begin{equation}
\label{eq:transmult}
\alxydim{@C=1.5cm}{\pr_1^{*}\inf G \otimes \pr_2^{*}\inf G \ar[r]^-{\tr_{\mathcal{M}}} & m^{*}\inf G \otimes \tr_{\mathcal{I}_{\rho}} \ar[r]^-{\id \otimes t_{\rho}} & m^{*}\inf G\text{.}}
\end{equation}
It defines a smooth map $\tilde m\maps  \inf G \times \inf G \to \inf G$ that covers the multiplication of $LG$ along the projection $\inf G \to LG$. The transgression of the associator $\alpha$ guarantees that $\tilde m$ is associative. Then, the principal $A$-bundle $\inf G$ together with this product becomes a central extension
\begin{equation*}
1 \to A \to \inf G \to LG \to 1
\end{equation*}
of Fréchet Lie groups  \cite[Theorem 3.1.7]{waldorf5}.

We notice that  $\inf G = \tr_{\mathcal{G}}$ is a fusion bundle, see \erf{eq:transgression}. The bundle morphism \erf{eq:transmult} is fusion-preserving, since the isomorphism $\mathcal{M}$ transgresses to a fusion-preserving bundle morphism, and the trivialization \erf{eq:cantriv} is fusion-preserving \cite[Lemma 3.6]{waldorf13}. This shows that the fusion product of $\inf G$ is multiplicative
in the sense of Definition \ref{def:fusionextension}. Summarizing, we have shown:

\begin{theorem}
\label{th:fusionextension}
Let $\mathcal{G}$ be a multiplicative $A$-bundle gerbe with connection over $G$. Then, $\tr_{\mathcal{G}}$ is a fusion extension of $LG$ by $A$.
\end{theorem}

Suppose that $G$ is a compact, connected, simple and simply-connected Lie group. Let $\generator \in \h^3(G,\Z)$  be a  generator. There is a canonical multiplicative bundle gerbe with connection, the \emph{basic bundle gerbe} $\gbas$, whose Dixmier-Douady class is $\generator$, see \cite[Example 1.5]{waldorf5} and \cite{Waldorf}.  Since
\begin{equation*}
\mathrm{c}_1(\tr_{\gbas}) \stackerf{eq:transgressionclasses}{=} -\tau(\mathrm{dd}(\gbas)) = -\tau(\generator)
\end{equation*}  
we have according to our sign convention \erf{eq:signconvention}:

\begin{corollary}
\label{co:fusionuniversal}
Suppose that $G$ is a compact, connected, simple and simply-connected Lie group. Then, 
$\tr_{\gbas}$ is the universal central extension $\widetilde{LG}$. In particular, the universal central extension of $LG$ is a fusion extension.
\end{corollary}

In the remainder of this section we show that the abstractly defined fusion extension $\tr_{\mathcal{G}}$ and the explicitly constructed fusion extension $\wzwmodel$ from Section \ref{sec:fusionspinstructures} are isomorphic, under the assumption that $(\mathcal{G},\rho,\mathcal{M},\alpha)$ is a multiplicative bundle gerbe  with connection over a compact, simple, connected, simply-connected Lie group $G$, and $\wzwmodel$ is constructed using the  forms $H\df\mathrm{curv}(\mathcal{G})$ and $\rho$. Note that the two identities \erf{eq:id1} and \erf{eq:id2} from the multiplicative structure are precisely those required in the construction of $\inf L$.

The isomorphism between $\inf L$ and $\tr_{\mathcal{G}}$ is defined by
\begin{equation*}
\varphi \maps  \inf L \to \tr_{\mathcal{G}}\maps  (\phi,z) \mapsto \partial\mathcal{T} \cdot z \cdot \exp \left ( -2\pi\im \int_{D^2} \omega \right )\text{.}
\end{equation*}
Here, $\mathcal{T}\maps  \phi^{*}\mathcal{G} \to \mathcal{I}_{\omega}$ is an arbitrarily chosen trivialization of $\phi^{*}\mathcal{G}$ over $D^2$ and $\partial \mathcal{T}$ denotes its restriction to the boundary; the latter is a trivialization of $\partial\phi^{*}\mathcal{G}$ over $S^1$, i.e. an element in $\tr_{\mathcal{G}}$ over the loop $\partial\phi$, see \erf{eq:lgfibre}. The map $\varphi$ is evidently fibre-preserving and $\ueins$-equivariant; hence, a bundle isomorphism. 

Let us show that $\varphi$ is well-defined. 
If $\mathcal{T}'\maps  \phi^{*}\mathcal{G} \to \mathcal{I}_{\omega'}$ is a different choice of a trivialization, then there exists a principal $\ueins$-bundle $P$ with connection over $D^2$ with $\mathrm{curv}(P) = \omega' - \omega$ and $\partial\mathcal{T}' = \partial\mathcal{T} \cdot \mathrm{Hol}_{P}(S^1)$ as elements in $\tr_{\mathcal{G}}|_{\partial\phi}$, see \cite[Section 3.1]{waldorf5} for some more details. Since
\begin{equation*}
\mathrm{Hol}_P(\partial D^2) = \exp \left ( 2\pi\im \int_{D^2} \mathrm{curv}(P) \right )\text{,}
\end{equation*}
the two contributions cancel. 
Suppose, on the other hand,  that $(\phi',z')$ is an equivalent representative of the element in $\wzwmodel$, i.e., $\partial\phi = \partial\phi'$
and $z = z' \cdot \mathrm{e}^{2\pi\im S_{\mathrm{WZ}}(\Phi)}$. We may choose a trivialization $\widetilde{\mathcal{T}}\maps  \Phi^{*}\mathcal{G} \to \mathcal{I}_{\tilde \omega}$ and use its restrictions $\mathcal{T}$ and $\mathcal{T}'$  to the two hemispheres in the definition of $\varphi$. Now, well-definedness follows from the fact that
\begin{equation*}
\exp \left (2\pi\im S_{\mathrm{WZ}}(\Phi) \right ) = \exp \left (2\pi\im \int_{D^3} \tilde \Phi^{*}H \right ) = \exp\left ( 2\pi\im \int_{S^2} \tilde \omega \right )\text{,} 
\end{equation*}
the latter equality being a consequence of the integrality of $H$ and the relation $\Phi^{*}H \eq \mathrm{curv}(\Phi^{*}\mathcal{G}) \eq \mathrm{d}\tilde \omega$ (as for any trivialization).

\begin{proposition}
\label{prop:equivext}
The bundle isomorphism $\varphi$ is a group homomorphism and  fusion-preserving, and thus defines an equivalence between the fusion extensions $\tr_{\mathcal{G}}$ and $\wzwmodel$. 
\end{proposition}

\begin{proof}
In order to see that $\varphi$ preserves the group structure, consider two elements $(\phi_1,z_1)$ and $(\phi_2,z_2)$ in $\wzwmodel$, and recall that 
\begin{equation*}
(\phi_1,z_1) \cdot (\phi_2,z_2) = (\phi_1\phi_2, z_1z_2 \cdot \exp \left (- 2\pi\im \int_{D^2} (\phi_1,\phi_2)^{*}\rho \right )  )\text{.}
\end{equation*}
Let $\mathcal{T}_1\maps  \phi_1^{*}\mathcal{G} \to \mathcal{I}_{\omega_1}$ and $\mathcal{T}_2\maps  \phi_2^{*}\mathcal{G} \to \mathcal{I}_{\omega_2}$ be trivializations. We define a new trivialization $\mathcal{T}_{12}\maps (\phi_1\phi_2)^{*}\mathcal{G} \to \mathcal{I}_{\omega_{12}}$ as the composite
\begin{equation*}
\alxydim{@C=1.5cm@R=1.6cm}{(\phi_1\phi_2)^{*}\mathcal{G} \ar@{=}[r] & (\phi_1,\phi_2)^{*}m^{*}\mathcal{G} \otimes \mathcal{I}_{(\phi_1,\phi_2)^{*}\rho} \otimes \mathcal{I}_{-(\phi_1,\phi_2)^{*}\rho} \ar[d]^-{(\phi_1,\phi_2)^{*}\mathcal{M}^{-1} \otimes \id} \\& \phi_1^{*}\mathcal{G} \otimes \phi_2^{*}\mathcal{G}  \mathcal{I}_{-(\phi_1,\phi_2)^{*}\rho} \ar[r]_{\mathcal{T}_1 \otimes \mathcal{T}_2 \otimes \id} &  \mathcal{I}_{\omega_1 + \omega_2 - (\phi_1,\phi_2)^{*}\rho} }
\end{equation*}
Restricting to the boundary $S^1 = \partial D^2$, we have $\partial\mathcal{T}_{12} = \partial\mathcal{T}_1 \cdot \partial\mathcal{T}_2$ in the group structure of the Fréchet Lie group $\tr_{\mathcal{G}}$, compare \erf{eq:transmorph} and \erf{eq:transmult}. Hence,
\begin{eqnarray*}
\varphi(\phi_1,z_1) \cdot \varphi(\phi_2,z_2) &=& \partial\mathcal{T}_{1} \cdot \partial\mathcal{T}_{2} \cdot z_1z_2 \cdot \exp \left ( - 2\pi\im \int_{D^2} \omega_1 + \omega_2 \right ) 
\\
&=&
\partial\mathcal{T}_{12} \cdot z_1z_2 \cdot \exp \left ( - 2\pi\im \int_{D^2} \omega_{12} + (\phi_1,\phi_2)^{*}\rho \right ) 
\\
&=&
\varphi(\phi_1\phi_2,z_1z_2) \cdot \exp \left ( - 2\pi\im \int_{D^2} (\phi_1,\phi_2)^{*}\rho \right ) 
\\
&=&
\varphi((\phi_1,z_1) \cdot (\phi_2,z_2))\text{.}
\end{eqnarray*}
This shows that $\varphi$ is a group homomorphism. 

In order to see that $\varphi$ is fusion-preserving we assume that $(\gamma_1,\gamma_2,\gamma_3)\in PG^{[3]}$, that $\phi_{ij}\maps  D^2 \to G$ are smooth maps with $\partial\phi_{ij}=\gamma_i \lop \gamma_j$, and that  $\Psi\maps S^2 \to G$ is constructed from the latter ones as described in Section \ref{sec:fusionspinstructures}, so that
\begin{equation*}
\lambda_{\wzwmodel}((\phi_{12},z_{12}) \otimes (\phi_{23},z_{23})) = (\phi_{13},z_{12}  z_{23}\cdot \mathrm{e}^{2\pi\im S_{\mathrm{WZ}}(\Psi)})\text{.}
\end{equation*}
Let us choose a trivialization $\mathcal{T}\maps  \Psi^{*}\mathcal{G} \to \mathcal{I}_{\omega}$, and restrict to trivializations $\mathcal{T}_{ij}\maps \phi_{ij}^{*}\mathcal{G} \to \mathcal{I}_{\omega_{ij}}$. In this situation, Definition \erf{eq:fusiondef} of the fusion product $\lambda_{\mathcal{G}}$ on $\tr_{\mathcal{G}}$ shows that 
\begin{equation*}
\lambda_{\mathcal{G}}(\partial\mathcal{T}_{12} \otimes \partial\mathcal{T}_{23}) = \partial\mathcal{T}_{13}\text{.}
\end{equation*}
Thus,
\begin{eqnarray*}
\lambda_{\mathcal{G}}(\varphi(\phi_{12},z_{12}) \otimes \varphi(\phi_{23},z_{23})) &=& \partial\mathcal{T}_{13} \cdot z_{12}z_{23} \cdot \exp \left ( - 2\pi\im \int_{D^2} (\omega_{12} + \omega_{23}) \right ) 
\\
&=& \partial\mathcal{T}_{13} \cdot z_{12}z_{23}\cdot \exp \left ( - 2\pi\im \int_{D^2} (\omega_{13} - \omega) \right ) 
\\
&=& \partial\mathcal{T}_{13} \cdot z_{12}z_{23}\cdot \mathrm{e}^{2\pi\im S_{\mathrm{WZ}}(\Psi)}\cdot \exp \left ( - 2\pi\im \int_{D^2} \omega_{13}  \right ) 
\\
&=& \varphi(\phi_{13},z_{12} z_{23}\cdot \mathrm{e}^{2\pi\im S_{\mathrm{WZ}}(\Psi)})
\text{.}
\end{eqnarray*}
This shows that $\varphi$ is fusion-preserving. 
\end{proof}

\subsection{Fusion Lifts}

\label{sec:liftingfusionext}

In this section we  develop a general theory for lifting problems along fusion extensions, which is a general framework for the fusion spin structures of Definition \ref{def:fusionspinstructure}.

We consider a connected Lie group $G$, an abelian Lie group $A$,  a fusion extension
\begin{equation*}
1 \to A \to \inf G \to LG \to 1\text{,}
\end{equation*}
and a principal $G$-bundle $E$ over a smooth manifold $M$. The fusion product on $\inf G$ will be denoted by $\lambda_{\inf G}$. We recall that $LE$ is a Fréchet principal $LG$-bundle over $LM$. 

\begin{definition}
\label{def:fusionlift}
A \emph{fusion lift} of the structure group of $LE$ from $LG$ to $\inf G$ is a $\inf G$-lift $(\inf S,\sigma)$  together with a fusion product $\lambda_{\inf S}$ on the associated principal $A$-bundle $\sigma\maps  \inf S \to LE$, such that the $\inf G$-action on $\inf S$ is fusion-preserving:
\begin{equation*}
\lambda_{\inf E}(q_{12}\cdot \beta_{12} \otimes q_{23} \cdot \beta_{23}) = \lambda_{\inf E}(q_{12} \otimes q_{23}) \cdot \lambda_{\inf G}(\beta_{12} \otimes \beta_{23})
\end{equation*}
for all $q_{ij}\in \inf S_{\alpha_i\lop\alpha_j}$, $\beta_{ij} \in \inf G_{\gamma_i \lop \gamma_j}$, and all $(\alpha_1,\alpha_2,\alpha_3) \in PE^{[3]}$ and $(\gamma_1,\gamma_2,\gamma_3) \in PG^{[3]}$. 
\end{definition}

Here, by \emph{associated $A$-bundle} we mean the bundle defined analogously to the one of Lemma \ref{lem:ueinsbundle}, which has the total space $\inf S$ and the $A$ action induced by the group homomorphism
\begin{equation*}
\alxydim{}{A \ar[r]^-{i} & A \ar@{^(->}[r] & \inf G\text{.} }
\end{equation*}
A morphism between fusion lifts $(\inf S_1,\sigma_1,\lambda_{\inf S_1})$ and $(\inf S_2,\sigma_2,\lambda_{\inf S_2})$ is called \emph{fusion-preserving}, if the induced morphism $\inf S_1 \to \inf S_2$ of principal $A$-bundles over $LE$ preserves the fusion products $\lambda_1$ and $\lambda_2$. Fusion lifts form a category denoted by $\fusstruc {\inf G}{LE}$.

\begin{example}
\label{ex:fusionspinstructures}
In this notation,  the fusion spin structures introduced in Definition \ref{def:fusionspinstructure} form precisely the category $\fusstruc {\lspinhat n}{LFM}$. 
\end{example}

In the following we describe fusion lifts in terms of lifting bundle gerbes, and equip, for this purpose, the ordinary lifting bundle gerbe $\mathcal{G}_{LE}$ for lifting the looped bundle $LE$ from $LG$ to $\inf G$ with an additional  structure.
\begin{definition}
\label{def:internalfusionproduct}
Let $\mathcal{G}$ be an $A$-bundle gerbe over $LM$ whose surjective submersion is the looping of a surjective submersion $\pi\maps Y \to M$. Let  $P$ denote its principal $A$-bundle over $LY^{[2]}$ and $\mu$ denote its bundle gerbe product.  An \emph{internal fusion product} on $\mathcal{G}$ is a fusion product on  $P$ such that  $\mu$ is fusion-preserving. 
\end{definition}

The condition that $\mu$ is fusion-preserving makes sense since $\mu$ is a morphism
\begin{equation*}
\mu\maps  \pr_{12}^{*}P \otimes \pr_{23}^{*}P \to \pr_{13}^{*}P
\end{equation*}
between fusion bundles over $LY^{[3]}$. 
In order to spell it out explicitly, we must consider three paths
 $\gamma_1,\gamma_2,\gamma_3 \in PY^{[3]}$ with a common initial point in $Y^{[3]}$ and a common end point in $Y^{[3]}$. The composition of a path $\gamma_k$ with one of the projections $\pr_{i}\maps  Y^{[3]} \to PY$ gives a path $\gamma_k^{i} \in PY$, and these give in turn loops which we denote by $\tau_{ij} \df \gamma^{1}_i \lop \gamma^{1}_j \in LY$, $\tau'_{ij}= \gamma^{2}_i \lop \gamma^{2}_j \in LY$, and $\tau''_{ij}= \gamma^{3}_i \lop \gamma^{3}_j \in LY$. Now, the condition is that
\begin{equation*}
\mu(\lambda(q_{12} \otimes q_{23}) \otimes \lambda(q_{12}' \otimes q_{23}')) = \lambda(\mu(q_{12} \otimes q_{12}') \otimes \mu(q_{23} \otimes q_{23}'))
\end{equation*}
for all $q_{ij} \in P_{(\tau_{ij},\tau_{ij}')}$ and $q_{ij}' \in P_{(\tau_{ij}',\tau_{ij}'')}$.

There are two important examples of bundle gerbes with internal fusion products. One is when the $\mathcal{G}$ is the transgression of a bundle 2-gerbe over $M$ -- this will be explained in Section \ref{sec:bundle2gerbes}. The other example is  when $\mathcal{G}$  is the lifting bundle gerbe $\mathcal{G}_{LE}$  associated to the problem of lifting the structure group of $LE$ from $LG$ to the fusion extension $\inf G$.

Indeed, in this case the surjective submersion of $\mathcal{G}_{LE}$ is the looping of the bundle projection $E \to M$, and
the principal $A$-bundle $P=L\diffr^{*}\inf G$ is the pullback of a fusion bundle along a looped map, and thus a fusion bundle, see \erf{eq:fusionpullback}. We will denote this internal fusion product on $\mathcal{G}_{LE}$ by $\lambda_{LE}$. 

In order to check that the bundle gerbe product  \erf{bgprod} is fusion-preserving, we consider paths $\gamma_1,\gamma_2,\gamma_3 \in PE^{[3]}$ as above
, and write $q_{ij} = (\tau_{ij},\tau_{ij}',\beta_{ij})$ and $q_{ij}' = (\tau'_{ij},\tau_{ij}'',\beta'_{ij})$ for elements $\beta_{ij} \in \inf G_{\gamma_{ij}}$ and $\beta_{ij}'\in \inf G_{\gamma_{ij}'}$, where $\gamma_{ij} \df L\diffr(\tau_{ij},\tau_{ij}') \in LG$ and $\gamma_{ij}' \df L\diffr(\tau_{ij}',\tau_{ij}'') \in LG$. Note that now
\begin{equation}
\label{eq:exprpull}
\lambda_{LE}(q_{12} \otimes q_{23}) = \lambda_{LE}((\tau_{12},\tau_{12}',\beta_{12}) \otimes (\tau_{23},\tau_{23}',\beta_{23})) = (\tau_{13},\tau_{13}',\lambda_{\inf G}(\beta_{12} \otimes \beta_{23}))\text{,}
\end{equation}
and similarly for the primed elements.
Then we calculate:
\begin{eqnarray*}
&&\hspace{-3cm}\mu(\lambda_{LE}(q_{12} \otimes q_{23}) \otimes \lambda_{LE}(q_{12}' \otimes q_{23}'))\\ &\stackerf{eq:exprpull}{=}&\mu((\tau_{13},\tau_{13}',\lambda_{\inf G}(\beta_{12} \otimes \beta_{23})) \otimes (\tau'_{13},\tau_{13}'',\lambda_{\inf G}(\beta_{12}' \otimes \beta_{23}')))
\\&\stackerf{bgprod}{=}&
(\tau_{13},\tau_{13}'',\lambda_{\inf G}(\beta_{12}' \otimes \beta_{23}') \cdot \lambda_{\inf G}(\beta_{12} \otimes \beta_{23}))
\\&\stackerf{eq:multfus}{=}&
(\tau_{13},\tau_{13}'',\lambda_{\inf G}(\beta_{12}'\beta_{12} \otimes \beta_{23}'\beta_{23}) )
\\&\stackerf{eq:exprpull}{=}&
\lambda_{LE}((\tau_{12},\tau_{12}'',\beta_{12}'\beta_{12}) \otimes (\tau_{23},\tau_{23}'',\beta_{23}'\beta_{23}))
\\&\stackerf{bgprod}{=}&
\lambda_{LE}(\mu(q_{12} \otimes q_{12}') \otimes \mu(q_{23} \otimes q_{23}'))\text{.}
\end{eqnarray*}
This shows that $\mu$ is fusion-preserving.
Summarizing, we have defined an internal fusion product on the lifting bundle gerbe $\mathcal{G}_{LE}$.

\begin{example}
\label{ex:spinliftingfusion}
The spin lifting gerbe $\mathcal{S}_{LM}$ over the loop space of a spin manifold (see Example \ref{ex:spinliftinggerbe}) is equipped with an internal fusion product.
\end{example}

\begin{remark}
In contrast to an internal fusion product, which is only defined for a particular class of bundle gerbes over $LM$,
an \emph{external fusion product} on a general bundle gerbe $\mathcal{G}$ over  $LM$ is an isomorphism
\begin{equation*}
\mathcal{L}\maps  e_{12}^{*}\mathcal{G} \otimes e_{23}^{*}\mathcal{G} \to e_{13}^{*}\mathcal{G}
\end{equation*}
between bundle gerbes over $PM^{[3]}$ together with a transformation
\begin{equation*}
\alxydim{@C=2cm@R=1.6cm}{e_{12}^{*}\mathcal{G} \otimes e_{23}^{*}\mathcal{G} \otimes e_{34}^{*}\mathcal{G} \ar[r]^-{\id \otimes \pr_{234}^{*}\mathcal{L}} \ar[d]_{\pr_{123}^{*}\mathcal{L} \otimes \id} & e_{12}^{*}\mathcal{G} \otimes e_{24}^{*}\mathcal{G} \ar@{=>}[dl]|*+{\lambda} \ar[d]^{\pr_{124}^{*}\mathcal{L}} \\ e_{13}^{*}\mathcal{G} \otimes e_{34}^{*}\mathcal{G} \ar[r]_-{\pr_{134}^{*}\mathcal{L}} & e_{14}^{*}\mathcal{G}}
\end{equation*}
over $PM^{[4]}$ that satisfies the pentagon identity over $\p M^{[5]}$. Here, $e_{ij}\maps  PM^{[2]} \to LM$ is the composition of the projection $\pr_{ij}\maps  PM^{[3]} \to PM^{[2]}$ with the map $\lop\maps  PM^{[2]} \to LM$. I remark that an internal fusion product determines an external one;  a general discussion of external fusion products is beyond the scope of this article.  
\end{remark}

If a bundle gerbe $\mathcal{G}$ is equipped with an internal fusion product $\lambda$, we can  consider trivializations that \quot{respect} the fusion product in the following way:

\begin{definition}
Let $\mathcal{G}$ be a bundle gerbe over $LM$ whose surjective submersion is the looping of a surjective submersion $\pi\maps Y \to M$. A \emph{fusion product} on a trivialization $\mathcal{T} = (T,\kappa)$ of $\mathcal{G}$ is a fusion product $\lambda_T$ on the principal $A$-bundle $T$ over $LY$. It is called \emph{compatible} with an internal fusion product $\lambda$ on $\mathcal{G}$ if $\kappa$ is a fusion-preserving bundle morphism. \end{definition}

A morphism $\varphi\maps  (T_1,\kappa_1,\lambda_{1})\to (T_2,\kappa_2,\lambda_2)$ between trivializations with compatible fusion products is an ordinary morphism $\varphi\maps  (T_1,\kappa_1) \to (T_2,\kappa_2)$ between the trivializations that is additionally fusion-preserving.  
The category of trivializations with compatible fusion product is denoted by $\trivfus {\mathcal{G},\lambda}$.

\begin{theorem}
\label{th:fusionlifting}
The equivalence of Theorem \ref{th:lifting} between trivializations of $\mathcal{G}_{LE}$ and $\inf G$-lifts of $LE$ extends to an equivalence in the fusion setting:
\begin{equation*}
\trivfus{\mathcal{G}_{LE},\lambda_{LE}} \cong \fusstruc{\inf G}{LE}\text{.}
\end{equation*}
\end{theorem}

\begin{proof}
We recall that  the equivalence of Theorem \ref{th:lifting} sends a trivialization $(T,\kappa)$ to the principal $\inf G$-bundle $\inf E \df T$ with projection $T \to LE \to LM$ and $\inf G$-action given by 
\begin{equation*}
q \cdot \beta \df \kappa(\beta^{-1} \otimes q)\text{.}
\end{equation*}
The \emph{additional structure} we want to take into account is the same on both sides: a fusion product $\lambda_{T}$ on the principal $A$-bundle $T \to LE$. It remains to check that the \emph{conditions} are equivalent: on the left hand side the condition that $\kappa$ is fusion-preserving,  and on the right hand side the condition that the $\inf G$-action on $T$ is fusion-preserving in the sense of Definition \ref{def:fusionlift}. 
Suppose first that $\kappa$ is fusion-preserving. Then,
\begin{eqnarray*}
\lambda_{\inf E}(q_{12}\cdot \beta_{12} \otimes q_{23} \cdot \beta_{23}) &=& 
\lambda_{\inf E}(\kappa(\beta_{12}^{-1} \otimes q_{12})\otimes \kappa(\beta_{23}^{-1} \otimes q_{23})) 
\\
&=&
\kappa(\lambda_{\inf G}(\beta_{12}^{-1} \otimes \beta_{23}^{-1}) \otimes \lambda_{\inf E}(q_{12} \otimes q_{23}))
\\
&=&
\kappa(\lambda_{\inf G}(\beta_{12} \otimes \beta_{23})^{-1} \otimes \lambda_{\inf E}(q_{12} \otimes q_{23}))
\\
&=&
\lambda_{\inf E}(q_{12} \otimes q_{23}) \cdot \lambda_{\inf G}(\beta_{12} \otimes \beta_{23})\text{;}
\end{eqnarray*}
this shows that the $\inf G$-action is fusion-preserving. 
In the middle we  have used that the multiplicativity \erf{eq:multfus} of $\lambda_{\inf G}$ implies $\lambda_{\inf G}(\beta_{12} \otimes \beta_{23})^{-1}= \lambda_{\inf G}(\beta_{12}^{-1} \otimes \beta_{23}^{-1})$.
The converse can be proved similarly. 
The condition for morphisms is evidently the same on both sides.
\end{proof}

Applied to fusion spin structures on loop spaces, Theorem \ref{th:fusionlifting} implies via Examples \ref{ex:fusionspinstructures} and  \ref{ex:spinliftingfusion}:

\begin{corollary}
\label{co:fusionspinlifting}
There is an equivalence
\begin{equation*}
\bigset{8em}{Fusion spin structures on $LM$}
\cong
\bigset{11em}{Trivializations of the spin lifting gerbe $\mathcal{S}_{LM}$ with compatible fusion product}\text{.}
\end{equation*}
\end{corollary}

\setsecnumdepth 2

\section{Transgression of String Structures}

\label{sec:transgressionstringstructures}

In this section we prepare another important tool for the proof of our main result: we discuss string structures in the setting of bundle 2-gerbes. 

\subsection{Bundle 2-Gerbes and String Structures}

\label{sec:trivializationsof2gerbes}
\label{sec:bundle2gerbes}

We start with recalling some basic definitions.

\begin{definition}[{{\cite[Definition 5.3]{stevenson2}}}]
\label{def:bundle2gerbe}
A \emph{bundle 2-gerbe} over $M$ is a surjective submersion $\pi \maps Y \to M$ together with a bundle gerbe $\mathcal{P}$  over $Y^{[2]}$, an isomorphism 
\begin{equation*}
\mathcal{M}\maps  \pr_{12}^{*}\mathcal{P} \otimes \pr_{23}^{*}\mathcal{P} \to \pr_{13}^{*}\mathcal{P}
\end{equation*}
of bundle gerbes over $Y^{[3]}$, and a transformation 
\begin{equation*}
\alxydim{@C=2cm@R=1.6cm}{\pr_{12}^{*}\mathcal{P} \otimes \pr_{23}^{*}\mathcal{P} \otimes \pr_{34}^{*}\mathcal{P} \ar[r]^-{\pr_{123}^{*}\mathcal{M} \otimes \id} \ar[d]_{\id \otimes \pr i_{234}^{*}\mathcal{M}} & \pr_{13}^{*}\mathcal{P} \otimes \pr_{34}^{*}\mathcal{P} \ar@{=>}[dl]|*+{\mu} \ar[d]^{\pr_{134}^{*}\mathcal{M}} \\ \pr_{12}^{*}\mathcal{P} \otimes \pr_{24}^{*}\mathcal{P} \ar[r]_-{\pr_{124}^{*}\mathcal{M}} & \pr_{14}^{*}\mathcal{P}}
\end{equation*}
over $Y^{[4]}$ that satisfies the  pentagon axiom.
\end{definition}

The isomorphism $\mathcal{M}$ is called \emph{bundle 2-gerbe product} and the transformation $\mu$ is called  \emph{associator}. The pentagon axiom implies the cocycle condition for a certain degree three cocycle on $M$ with values in $\ueins$, which defines -- via the exponential sequence -- a class 
\begin{equation*}
\mathrm{cc}(\mathbb{G}) \in \h^4(M,\Z)\text{;}
\end{equation*}
see \cite[Proposition 7.2]{stevenson2} for the details.

\begin{definition}[{{\cite[Definition 11.1]{stevenson2}}}]
\label{deftriv}
Let $\mathbb{G}=(Y,\pi,\mathcal{P},\mathcal{M},\mu)$ be a bundle 2-gerbe over $M$. A \emph{trivialization} of $\mathbb{G}$ is a bundle gerbe $\mathcal{S}$ over $Y$, together with an  isomorphism
\begin{equation*}
\mathcal{A}\maps  \mathcal{P} \otimes \pr_2^{*}\mathcal{S} \to \pr_1^{*}\mathcal{S}
\end{equation*}
of bundle gerbes over $Y^{[2]}$ and a connection-preserving transformation
\begin{equation*}
\alxydim{@C=2cm@R=1.6cm}{\pr_{12}^{*}\mathcal{P} \otimes \pr_{23}^{*}\mathcal{P} \otimes \pr_3^{*}\mathcal{S} \ar[r]^-{\id \otimes \pr_{23}^{*}\mathcal{A}} \ar[d]_{\mathcal{M} \otimes \id} & \pr_{12}^{*}\mathcal{P} \otimes \pr_{2}^{*}\mathcal{S} \ar@{=>}[dl]|*+{\sigma} \ar[d]^{\pr_{12}^{*}\mathcal{A}} \\ \pr_{13}^{*}\mathcal{P} \otimes \pr_{3}^{*}\mathcal{S} \ar[r]_-{\pr_{13}^{*}\mathcal{A}} & \pr_1^{*}\mathcal{S}}
\end{equation*}
over $Y^{[3]}$ that satisfies a compatibility condition with the associator $\mu$.
\end{definition}

As one expects, the characteristic class $\mathrm{cc}(\mathbb{G}) \in \h^4(M,\Z)$ of $\mathbb{G}$ vanishes if and only if $\mathbb{G}$ admits a trivialization \cite[Proposition 11.2]{stevenson2}. 
An example of a bundle 2-gerbe that will be important later is the \emph{Chern-Simons 2-gerbe} $\mathbb{CS}_P(\mathcal{G})$ \cite{carey4} -- it is associated to a principal $G$-bundle $P$ over $M$ and a multiplicative bundle gerbe $\mathcal{G}$ over $G$. Its construction goes as follows: 
\begin{itemize}

\item 
The surjective submersion is the bundle projection $P \to M$.  

\item
The bundle gerbe $\mathcal{P}$ over $P^{[2]}$ is $\mathcal{P} \df         \diffr^{*}\mathcal{G}$, where $\diffr$ is the difference map \erf{eq:differencemap}.
\item
We consider the map $\diffr'\maps  P^{[3]} \to G^2$ defined by $(p,p') \cdot \diffr'(p,p',p'') = (p',p'')$, and obtain the required bundle 2-gerbe product by pullback of the multiplicative structure \erf{eq:multstr}:
\begin{equation*}
\diffr^{\prime\ast}\mathcal{M}\maps  \pr_{12}^{*}\mathcal{P} \otimes \pr_{23}^{*}\mathcal{P} \to \pr_{13}^{*}\mathcal{P}\text{.}
\end{equation*}

\item
The transformation $\alpha$ gives via pullback along the analogous map $\diffr''\maps  P^{[4]} \to G^3$ the  associator. 
\end{itemize}

In this article, we will use the Chern-Simons 2-gerbe in order to obtain a geometrical notion of string structures on a spin manifold $M$. For this purpose, we consider $P=FM$, the spin-oriented frame bundle of $M$, and $\mathcal{G} = \gbas$, the basic bundle gerbe over $\spin n$, whose Dixmier-Douady class represents the fixed generator $\generator \in \h^3(\spin n,\Z)$. We write $\mathbb{CS}_M \df \mathbb{CS}_P(\gbas)$ for simplicity. We have \cite[Theorem 1.1.3]{waldorf8}:
\begin{equation*}
\mathrm{cc}(\mathbb{CS}_M) = \frac{1}{2}p_1(M)\text{.}
\end{equation*}
In particular, we see that $M$ is a string manifold if and only if $\mathbb{CS}_M$ admits a trivialization. This motivates the following definition:

\begin{definition}[{{\cite[Definition 1.1.5]{waldorf8}}}]
\label{def:stringstructure}
Let $M$ be a spin manifold. A \emph{string structure on $M$} is a
 trivialization of $\mathbb{CS}_M$. 
\end{definition}

So, a string structure on $M$ is a triple $(\mathcal{S},\mathcal{A},\sigma)$ consisting of a bundle gerbe $\mathcal{S}$ over $FM$, of an isomorphism $\mathcal{A}\maps  \diffr^{*}\gbas \otimes  \pr_2^{*}\mathcal{S} \to \pr_1^{*}\mathcal{S}$ between bundle gerbes over $FM^{[2]}$, and of a transformation $\sigma$ over $FM^{[3]}$.

Definition \ref{def:stringstructure} of a string structure has many nice features, some of which are described in \cite{waldorf8}. A particular feature, which we need in in the proof of the main theorem, is that Definition \ref{def:stringstructure} reproduces the topological notion of a \emph{string class}: a cohomology class $\xi\in \h^3(FM,\Z)$ that restricts on each fibre to the generator $\generator \in \h^3(\spin n,\Z)$.  Indeed, we have:

\begin{proposition}
[{{\cite[Theorem 1.1.4]{waldorf8}}}]
\label{th:stringclasses}
The map
\begin{equation*}
\bigset{3.9cm}{Isomorphism classes of string structures on $M$} \to 
\bigset{2.3cm}{String classes on $FM$}: (\mathcal{S},\mathcal{A},\sigma) \mapsto -\mathrm{dd}(\mathcal{S} )
\end{equation*}
that sends a trivialization to minus the Dixmier-Douady class of the bundle gerbe $\mathcal{S}$, is a bijection.
\end{proposition}

Let us recall one  aspect of Proposition \ref{th:stringclasses}, namely the fact that the map is well-defined. This follows from the following lemma, which will be used later.

\begin{lemma}
\label{lem:stringclass}
Let $\mathcal{S}$ be a bundle gerbe over $FM$, and let $\mathcal{A}\maps  \diffr^{*}\gbas \otimes \pr_2^{*}\mathcal{S} \to \pr_1^{*}\mathcal{S}$ be an isomorphism between bundle gerbes over $FM^{[2]}$. Then, $-\mathrm{dd}(\mathcal{S})$ is a string class.
\end{lemma}  

\begin{proof}
Let $p\in FM$ be a point and $\iota_p\maps  \spin n \to FM\maps  g \mapsto pg$ be the associated inclusion of the structure group of $FM$ in the fibre of $p$. Consider the map $s_p\maps  \spin n \to FM^{[2]}\maps g\mapsto (p,pg)$, which satisfies $\pr_2 \circ s_p = \iota_p$ and $\diffr \circ s_p = \id$, and for which $c_p \df \pr_1 \circ s_p$ is the constant map with value $p$. We obtain an isomorphism
\begin{equation*}
s_p^{*}\mathcal{A} \maps  \gbas \otimes \iota_p^{*}\mathcal{S} \to c_p^{*}\mathcal{S}\text{.}
\end{equation*}
Since the pullback of a bundle gerbe along a constant map is trivial, it has vanishing Dixmier-Douady class, we get
\begin{equation*}
\generator + \iota_p^{*}\mathrm{dd}(\mathcal{S}) = 0\text{.}
\end{equation*}
Hence, $-\mathrm{dd}(\mathcal{S})$ is a string class. 
\end{proof}

Another nice feature of Definition \ref{def:stringstructure} is that it provides a basis to include \emph{string connections}. The following definition summarizes the relevant notions concerning connections on bundle 2-gerbes.

\begin{definition}
\label{def:connection2gerbe}
\label{def:compatibleconnection}
Let $\mathbb{G}=(Y,\pi,\mathcal{P},\mathcal{M},\mu)$ be a bundle 2-gerbe over $M$. 
\begin{enumerate}[(i)]
\item 
A \emph{connection} on  $\mathbb{G}$ consists of a 3-form $C \in \Omega^3(Y)$, and of a connection on the bundle gerbe $\mathcal{P}$ such that
\begin{equation*}
\pi_2^{*}C - \pi_1^{*}C = \mathrm{curv}(\mathcal{P})\text{,}
\end{equation*}
and such that  $\mathcal{M}$ and $\mu$ is connection-preserving.

\item
Suppose $\mathbb{G}$ is equipped with a connection, and  $\mathbb{T}=(\mathcal{S},\mathcal{A},\sigma)$ is a trivialization of $\mathbb{G}$. A \emph{compatible connection} on $\mathbb{T}$ is a connection on the bundle gerbe $\mathcal{S}$ such that $\mathcal{A}$ and $\sigma$ are connection-preserving.

\end{enumerate}
\end{definition}

The Chern-Simons 2-gerbe $\mathbb{CS}_P(\mathcal{G})$ associated to a principal $G$-bundle $P$ over $M$ and a multiplicative bundle gerbe $\mathcal{G}$ over $G$ can be equipped with a connection that only depends on two parameters: a connection on the bundle $P$, and a multiplicative connection on the bundle gerbe $\mathcal{G}$, as recalled in Section \ref{sec:fusionextensions}, such that the two differential forms $H = \mathrm{curv}(\mathcal{G})$ and $\rho$ are the canonical ones, see \erf{eq:H} and \erf{eq:rho}. In the following we recall this construction on the basis of \cite[Section 3.2]{waldorf5}. 
\begin{itemize}
\item 
The 3-form is the Chern-Simons 3-form $CS(A) \in \Omega^3(P)$ associated to $A$. 

\item
The connection on the bundle gerbe $\mathcal{P}=\diffr^{*}\mathcal{G}$ over $P^{[2]}$ is given by the connection on $\diffr^{*}\mathcal{G}$ shifted by the 2-form
\begin{equation*}
\omega \df-  \left \langle   \diffr^{*} \bar\theta \wedge \pr_1^{*}A  \right \rangle \in
\Omega^2(P^{[2]})\text{,}
\end{equation*}
i.e. we have $\mathcal{P} = \diffr^{*}\mathcal{G} \otimes \mathcal{I}_{\omega}$ as bundle gerbes with connection. The reason for the shift is that otherwise the required identity for connections (see Definition \ref{def:connection2gerbe}) would not be satisfied. 

\item
That the bundle 2-gerbe product $\diffr^{\prime*}\mathcal{M}$ is connection-preserving for the shifted connections follows from the identity 
\begin{equation}
\label{eq:identityrhoomega}
\diffr^{\prime*}\rho + \pr_{12}^{*} \omega - \pr_{13}^{*}\omega + \pr_{23}^{*}\omega = 0\text{.}
\end{equation}
The associator is just the pullback of the connection-preserving transformation $\alpha$, and so connection-preserving.
\end{itemize}

In the situation of a spin manifold $M$, the Chern-Simons 2-gerbe $\mathbb{CS}_M=\mathbb{CS}_{FM}(\gbas)$ carries a canonical connection, defined by the Levi-Cevita connection on $M$ and the canonical connection on  the basic bundle gerbe $\gbas$, see \cite[Theorem 1.2.1]{waldorf9}.

\begin{definition}[{{\cite[Definition 1.2.2]{waldorf8}}}]
Let $\mathbb{T}$ be a string structure. A \emph{string connection} on $\mathbb{T}$ is a connection  on  $\mathbb{T}$ that is compatible with the canonical connection  on $\mathbb{CS}_M$. 
\end{definition}

This concept of a string connection reproduces \cite[Theorem 1.2.3]{waldorf8} the original definition of Stolz and Teichner \cite{stolz1}, which was formulated in terms of trivializations of Chern-Simons theory. One important statement about string connections that we will need later is the following theorem.

\begin{theorem}[{{\cite[Theorem 1.3.4]{waldorf8}}}]
\label{th:stringconnection}
Every string structure admits a string connection.
\end{theorem}

\subsection{Transgression}

\label{sec:transgression}
\label{sec:transgressionof2gerbes}

We define  transgression for (certain) bundle 2-gerbes and their trivializations. Let $\mathbb{G}$ be a bundle 2-gerbe with connection over $M$, consisting of a surjective submersion $\pi\maps Y \to M$, a curving $C \in \Omega^3(Y)$, a bundle gerbe $\mathcal{P}$ with connection over $Y^{[2]}$, and a connection-preserving bundle 2-gerbe product $\mathcal{M}$ with connection-preserving associator $\mu$.

The restriction we impose on the admissible bundle 2-gerbes is that their surjective submersion $\pi\maps Y \to M$ is \emph{loopable}, i.e., $L\pi\maps LY \to LY$ is again a surjective submersion. Then, we define the following  bundle gerbe $\tr_{\mathbb{G}}$ over $LM$. Its surjective submersion is $L\pi\maps LY \to LM$.
The principal $\ueins$-bundle over $LY^{[2]}$  is $\tr_{\mathcal{P}}$. Since transgression for bundle gerbes is functorial, natural and monoidal, the transgression $\tr_{\mathcal{M}}$ is a bundle gerbe product. It is associative since the associator $\mu$ of $\mathbb{G}$ transgresses to an equality.

\begin{proposition}
\label{prop:transgression2gerbesign}
On the level of characteristic classes, 
\begin{equation*}
\mathrm{dd}(\tr_{\mathbb{G}}) = - \tau(\mathrm{cc}(\mathbb{G}))\text{.}
\end{equation*}
\end{proposition}

\begin{proof}
The statement has no actual relevance  for this article, so that it may be enough to verify it up to torsion. This can be done by identifying a connection on the bundle gerbe $\tr_{\mathbb{G}}$: it comes from the  connections  on the transgressed bundles that we have ignored in the construction. Because of the sign in \erf{eq:transgressionclasses}, the  curving on $\tr_{\mathbb{G}}$ is $- \int_{S^1} \ev^{*}C$. Hence, the curvature of $\tr_{\mathbb{G}}$ is minus the transgression of the curvature of $\mathbb{G}$.
\end{proof}

As $\tr_{\mathcal{P}}$ is a fusion bundle and $\tr_{\mathcal{M}}$ is fusion-preserving, it is evident that the bundle gerbe $\tr_{\mathbb{G}}$ is equipped with an internal fusion product  (Definition \ref{def:internalfusionproduct}), which we denote by $\lambda_{\mathbb{G}}$.

\label{sec:transgressionoftrivializations}

If $\mathbb{T}=(\mathcal{S},\mathcal{A},\sigma)$ is a trivialization of $\mathbb{G}$ with compatible connection, we define a trivialization $\tr_{\mathbb{T}}$ of $\tr_{\mathbb{G}}$. The bundle gerbe $\mathcal{S}$ with connection transgresses to a principal $\ueins$-bundle $\tr_{\mathcal{S}}$ over $LY$. Since transgression is functorial, natural, and monoidal, the 1-morphism
\begin{equation*}
\mathcal{A}\maps  \mathcal{P} \otimes \pi_2^{*}\mathcal{S} \to \pi_1^{*}\mathcal{S}
\end{equation*}
over $Y^{[2]}$ transgresses to the required bundle morphism over $LY^{[2]}$, and the 2-isomorphism $\sigma$ implies the compatibility condition. 

Since $\tr_{\mathcal{S}}$ is a fusion bundle, the trivialization $\tr_{\mathbb{T}}$ carries a fusion product,  and since $\tr_{\mathcal{A}}$ is fusion-preserving, it is compatible with the internal fusion product $\lambda_{\mathbb{G}}$ of $\tr_{\mathbb{G}}$.
We obtain a functor
\begin{equation}
\label{eq:transgressionfunctor}
\tr \maps  \hc 1 \trivcon {\mathbb{G}} \to \trivfus{\tr_{\mathbb{G}},\lambda_{\mathbb{G}}}\text{.}
\end{equation} 
In the following we transgress the Chern-Simons 2-gerbe and use the functor
\erf{eq:transgressionfunctor} to transgress (geometric) string structures.

\begin{theorem}[{{\cite[Proposition 6.2.1]{Nikolausa}}}]
\label{th:liftingtransgression}
Let $\mathcal{G}$ be a multiplicative bundle gerbe with connection over a Lie group $G$, and let $P$ be a principal $G$-bundle over $M$ with connection. Let $\mathbb{CS}_P(\mathcal{G})$  be the associated Chern-Simons 2-gerbe with connection. Let $\mathcal{G}_{LP}$ be the lifting gerbe associated to the problem of lifting the structure group of $LP$ from $LG$ to the central extension $\tr_{\mathcal{G}}$. Then, there is a canonical isomorphism
\begin{equation*}
\varphi \maps  \tr_{\mathbb{CS}_P(\mathcal{G})} \to \mathcal{G}_{LP}
\end{equation*}
between bundle gerbes over $LM$.
\end{theorem}

The claimed isomorphism has been constructed in \cite[Proposition 6.2.1]{Nikolausa}; since we need it explicitly below we recall this construction.
\begin{itemize}
\item 
Both bundle gerbes have the same surjective submersion, $LP \to LM$. 

\item
Next we look at the principal $\ueins$-bundles over $LP^{[2]}.$   The one of $\tr_{\mathbb{CS}_P(\mathcal{G})}$ is $\tr_{\mathcal{P}}$, where $\mathcal{P} \eq \diffr^{*}\mathcal{G} \otimes \mathcal{I}_{\omega}$, while the one   of  $\mathcal{G}_{LP}$ is $P \df L\diffr^{*}\tr_{\mathcal{G}}$. Naturality of transgression and the canonical trivialization $t_{\omega}$ of $\tr_{\mathcal{I}_{\omega}}$  provide an isomorphism
\begin{equation*}
\alxydim{@C=2cm}{\tr_{\mathcal{P}} \cong L\diffr^{*}\tr_{\mathcal{G}} \otimes \tr_{\mathcal{I}_{\omega}} \ar[r]^-{\id \otimes t_{\omega}} & L\diffr^{*}\tr_{\mathcal{G}}=P} 
\end{equation*} 
between principal $\ueins$-bundles over $LP^{[2]}$ which gives our map $\varphi$. 

\item
Finally we have to look at the bundle gerbe products.
We claim that the diagram
\begin{equation*}
\alxydim{@=1.6cm}{\pr_{12}^{*}\tr_{\mathcal{P}} \otimes \pr_{23}^{*}\tr_{\mathcal{P}} \ar[d]_{\pr_{12}^{*}\varphi \otimes \pr_{23}^{*}\varphi} \ar[rr]^-{\tr_{\mathcal{M}'}} && \pr_{13}^{*}\tr_{\mathcal{P}} \ar[d]^{\pr_{13}^{*}\varphi} \\ \pr_{12}^{*}P \otimes \pr_{23}^{*}P \ar[rr]_-{\tilde m} && \pr_{13}^{*}P}
\end{equation*}
is commutative, which shows that the isomorphism $\varphi$ exchanges the bundle gerbe products. Indeed, substituting the  definitions gives the diagram
\begin{equation*}
\alxydim{@=1.6cm}{\ar[d]_{\id \otimes t_{\pr_{12}^{*}\omega+\pr_{23}^{*}\omega}}   \tr_{\diffr^{\prime *}(\pr_{1}^{*}\mathcal{G} \otimes \pr_2^{*}\mathcal{G})}
\otimes \tr_{\mathcal{I}_{\pr_{12}^{*}\omega + \pr_{23}^{*}\omega}}   \ar[rr]^-{\tr_{\diffr^{\prime *}\mathcal{M} \otimes
\id}} && \tr_{\diffr^{\prime *}m^{*}\mathcal{G}} \otimes \tr_{\mathcal{I}_{\pr_{13}^{*}\omega}} \ar[d]^{\id \otimes t_{\pr_{13}^{*}\omega}} 
\\
L\diffr'^{*}\tr_{\pr_{1}^{*}\mathcal{G} \otimes \pr_2^{*}\mathcal{G}} \ar[r]_-{L\diffr^{\prime *}\tr_{\mathcal{M}}} & L\diffr^{\prime *}(\tr_{m^{*}\mathcal{G} \otimes \mathcal{I}_{\rho}}) \ar[r]_-{\id \otimes t_{\diffr^{\prime*}\rho}}  & L\diffr^{\prime *}\tr_{m^{*}\mathcal{G}}}
\end{equation*}
which is commutative due to the identity \erf{eq:identityrhoomega}. 
\end{itemize}
Summarizing, we have constructed an isomorphism $\varphi\maps  \tr_{\mathbb{CS}_P(\mathcal{G})} \to \mathcal{G}_{LP}$ of bundle gerbes over $LM$.

In the situation of a spin manifold $M$ we obtain an isomorphism
\begin{equation*}
\varphi: \tr_{\mathbb{CS}_M} \to \mathcal{S}_{LM}
\end{equation*}
between the Chern-Simons 2-gerbe $\mathbb{CS}_M \eq \mathbb{CS}_{FM}(\gbas)$ and the spin lifting gerbe $\mathcal{S}_{LM}$, see Example \ref{ex:spinliftinggerbe}. With Proposition \ref{prop:transgression2gerbesign}, we obtain  on the level of characteristic classes:
\begin{equation*}
\textstyle\tau(\frac{1}{2}p_1(M))=\lambda_{LM}\text{.}
\end{equation*}
Next we come to a new, crucial property of the isomorphism $\varphi$.

\begin{proposition}
\label{prop:fusionpreserving}
The  isomorphism $\varphi$ of Theorem \ref{th:liftingtransgression} respects the internal fusion products.
\end{proposition}

\begin{proof}
We recall that the internal fusion product of $\tr_{\mathbb{CS}_P}(\mathcal{G})$ is the fusion product $\lambda_{\mathcal{P}}$ on the $\ueins$-bundle $\tr_{\mathcal{P}}$, while the fusion product of $\mathcal{G}_{LP}$ is the pullback of the fusion product $\lambda_{\mathcal{G}}$ of $\tr_{\mathcal{G}}$ along $L\diffr$. By  \cite[Lemma 3.6]{waldorf13} the trivialization $t_{\omega}$ is fusion-preserving. Thus, $\varphi$ is a composition of fusion-preserving isomorphisms, and hence fusion-preserving.
\end{proof}

\setsecnumdepth 1

\section{Proof of Theorem \ref{th:main}}

\label{sec:conclusion}

In this section we prove the main theorem of this article: the assertion that $M$ is string if and only if $LM$ is fusion spin.  
By Corollary \ref{co:fusionspinlifting}, $LM$ is fusion spin if and only if the spin lifting gerbe $\mathcal{S}_{LM}$ has a trivialization with compatible fusion product.

Suppose first that $M$ is string, so that there exists a string structure $\mathbb{T}$ (Definition \ref{def:stringstructure}). By Theorem \ref{th:stringconnection} $\mathbb{T}$ admits a string connection, together giving a trivialization of $\mathbb{CS}_M$ with compatible connection.  Its transgression is a trivialization $\tr_{\mathbb{T}}$ of $\tr_{\mathbb{CS}_M}$ with compatible fusion product, see Section \ref{sec:transgression}. Since the isomorphism $\tr_{\mathbb{CS}_M} \cong \mathcal{S}_{LM}$  of Theorem \ref{th:liftingtransgression} preserves the internal fusion products (Proposition \ref{prop:fusionpreserving}), $\tr_{\mathbb{T}}$ induces  a trivialization of  $\mathcal{S}_{LM}$ with compatible fusion product.

Conversely, suppose $\mathcal{S}_{LM}$ has a trivialization $(T,\kappa)$ with compatible fusion product $\lambda_T$. Let $p\in FM$ be a point. We may assume that $M$ is connected, otherwise we proceed with each connected component of $M$ separately. Then, $FM$ is also connected. Thus we have a well-defined regression functor
\begin{equation*}
\un_p\maps  \ufusbun {LFM} \to  \ugrb {FM}
\end{equation*}
and obtain a bundle gerbe $\mathcal{S} \df \un_p(T,\lambda_T)$ over $FM$. In $FM^{[2]}$ we choose the base point $(p,p)$, so that both projections $\pr_1,\pr_2\maps  FM^{[2]} \to FM$ are base point-preserving. Now, the fusion-preserving bundle morphism $\kappa$ over $LFM^{[2]}$ regresses to an isomorphism
\begin{equation*}
\un_{(p,p)}(P \otimes \pr_2^{*}T) \to \un_{(p,p)}(\pr_1^{*}T) 
\end{equation*}
between bundle gerbes over $FM^{[2]}$, where $P$ is the principal $\ueins$-bundle of the spin lifting gerbe. Going through its construction using the model $\lspinhat n = \tr_{\gbas}$, we find that   $P =L\diffr^{*}\tr_{\gbas}$. Then we use that there is a (canonical) isomorphism $\un_1(\tr_{\gbas}) \cong \mathcal{G}_{bas}$ (see Theorem \ref{th:transgressionregression}).  We get an isomorphism
\begin{equation*}
\un_{(p,p)}(P) = \un_{(p,p)}(L\diffr^{*}\tr_{\gbas})= \diffr^{*}\un_1(\gbas)\cong \diffr^{*}\gbas\text{.}
\end{equation*}
Using that regression is monoidal and functorial (see Lemma \ref{lem:regprop}), we end up with an isomorphism
\begin{equation*}
\mathcal{A}\maps \diffr^{*}\gbas \otimes \pr_2^{*}\mathcal{S} \to \pr_1^{*}\mathcal{S}\text{.}
\end{equation*}
By Lemma \ref{lem:stringclass}  $-\mathrm{dd}(\mathcal{S})$ is a string class; thus, $M$ is string.

\kobib{../../bibliothek/tex}

\end{document}